\documentclass[reprint,superscriptaddress,amsmath,amssymb,aps,prr,floatfix,longbibliography]{revtex4-2}
\usepackage{amsthm}
\usepackage{amsmath, mathtools}
\usepackage[utf8]{inputenc}	
\usepackage[english]{babel}
\usepackage{amsmath}
\usepackage{graphicx}		
\usepackage{natbib}
\usepackage{textcomp}
\usepackage{gensymb}
\usepackage[usenames,dvipsnames,svgnames,table]{xcolor}
\usepackage[hidelinks,colorlinks=false,urlcolor=Cerulean,citecolor=black]{hyperref}
\usepackage{siunitx}
\usepackage{mathrsfs}
\usepackage{multirow}
\usepackage{bm}
\usepackage{xfrac}
\usepackage{comment}
\usepackage[normalem]{ulem}

\renewcommand{\i}{\mathrm{i}}

\newtheorem{definition}{Definition}
\newtheorem{prop}{Proposition}

\newcommand{\revision}[1]{{\color{black}#1}}

\begin{document}

\title{Analytical prediction of specific spatiotemporal patterns in nonlinear oscillator networks with distance-dependent time delays}

\author{Roberto C. Budzinski}
\thanks{These authors contributed equally}
\author{Tung T. Nguyen}
\thanks{These authors contributed equally}
\author{Gabriel B. Benigno}
\thanks{These authors contributed equally}
\affiliation{Department of Mathematics, Western University, London, ON, Canada}
\affiliation{Brain and Mind Institute, Western University, London, ON, Canada}
\affiliation{Western Academy for Advanced Research, Western University, London, ON, Canada}
\author{Jacqueline Đo\`{a}n}
\affiliation{Department of Mathematics, Western University, London, ON, Canada}
\affiliation{Brain and Mind Institute, Western University, London, ON, Canada}
\affiliation{Western Academy for Advanced Research, Western University, London, ON, Canada}
\author{J\'an Min\'a{\v c}}
\affiliation{Department of Mathematics, Western University, London, ON, Canada}
\affiliation{Western Academy for Advanced Research, Western University, London, ON, Canada}
\author{Terrence J. Sejnowski}
\affiliation{The Salk Institute for Biological Studies, La Jolla, CA, USA}
\affiliation{Division of Biological Sciences, University of California at San Diego, La Jolla, CA, USA}
\author{Lyle E. Muller}\email{lmuller2@uwo.ca}
\affiliation{Department of Mathematics, Western University, London, ON, Canada}
\affiliation{Brain and Mind Institute, Western University, London, ON, Canada}
\affiliation{Western Academy for Advanced Research, Western University, London, ON, Canada}

\begin{abstract}
We introduce an analytical approach that allows predictions and mechanistic insights into the dynamics of nonlinear oscillator networks with heterogeneous time delays. We demonstrate that time delays shape the spectrum of a matrix associated to the system, leading to the emergence of waves with a preferred direction. We then create analytical predictions for the specific spatiotemporal patterns observed in individual simulations of time-delayed Kuramoto networks. This approach generalizes to systems with heterogeneous time delays at finite scales, which permits the study of spatiotemporal dynamics in a broad range of applications.
\end{abstract}

\maketitle

\section{Introduction}

What is the effect of heterogeneous time delays in networked systems? This question is difficult to treat analytically in the context of multiple distributed time delays. In recent work \cite{muller2016rotating}, we studied intracranial electrophysiological recordings from human clinical patients during sleep. We found that the 11-15 Hz sleep ``spindle'' oscillation, a brain rhythm important for learning and memory \cite{sejnowski2000we}, was not perfectly synchronized with zero phase difference across the cortex; rather, sleep spindles are organized into rotating waves that travel in a preferred direction (see \href{https://elifesciences.org/articles/17267#media1}{\textcolor{Cerulean}{Movie 1}} in \cite{muller2016rotating}). Importantly, the propagation speed of the observed waves is consistent with the axonal conduction speed of the long-range fiber network in the cortex (3-5 ms$^{-1}$ \cite{girard2001feedforward}). This set of observations raises an important question:~how do these fibers, with no major anisotropy, create a specific spatiotemporal structure with a preferred chirality?

In this work, we analyze a time-delay Kuramoto model to address this question. Utilizing a recently reported analytical approach to the Kuramoto dynamics \cite{muller2021algebraic}, we introduce a complex-valued delay operator. This operator shapes the dynamics of the Kuramoto system into waves traveling across the network. The combination of this delay operator and the adjacency matrix determines these dynamics through their effect on eigenvalues in the complex plane, thus providing mechanistic insights into the effect of heterogeneous time delays. The approach introduced here offers a mathematical description for the dynamics of time-delayed networks, an important open problem in physics \cite{atay2010complex} with many applications in neuroscience \cite{schroter2017micro}, engineering \cite{papachristodoulou2010effects}, and technology \cite{liao2000global}. In general, approaches to systems with heterogeneous time delays center on numerical simulations, and no coherent analytical approach currently exists \cite{tewarie2019spatially,petkoski2022normalizing}. Importantly, while this question first arose from observations of neural dynamics in the human cortex during sleep, the delay operator we introduce here is general to studying the effect of distributed time delays in networks at finite scales, potentially allowing insight into these dynamics in a broad range of systems \cite{lee2009large,papachristodoulou2006synchonization,roberts2019metastable}.

\section{Delay operator}

We start with the standard Kuramoto model (KM) \cite{kuramoto1975self, rodrigues2016kuramoto,acebron2005kuramoto} and then consider the model with distance-dependent time delays \cite{jeong2002time,ko2007effects, petkoski2016heterogeneity}. The original KM on a general network of $N$ nodes is defined by:
\begin{equation}
\dot{\theta}_i(t) = \omega_i + \epsilon \sum_{j=1}^{N} A_{ij} \sin( \theta_j(t) - \theta_i(t) )\,,
\label{eq:original_km_no_delay}
\end{equation}
where $\theta_i \in [-\pi,\pi)$ represents the state variable (phase) of oscillator $i$ at time $t$, $\omega_i$ is the intrinsic angular frequency, $\epsilon$ scales the coupling strength, and $A_{ij} \in \{0,1\}$ represents the elements of the adjacency matrix. The coupling of two connected oscillators $i$ and $j$ causes their phases to attract \cite{rodrigues2016kuramoto,acebron2005kuramoto,strogatz2000kuramoto,arenas2008synchronization}. 

Time delays have been observed to be an important mechanism underlying the generation of traveling waves in the brain \cite{muller2018cortical, davis2021spontaneous, roberts2019metastable, breakspear2010generative}. With this in mind, we consider a time-delay Kuramoto model (dKM) with delays $\tau_{ij}$ that depend on the distance between two oscillators $i$ and $j$:
\begin{equation}
\dot{\theta}_i(t) = \omega_i + \epsilon \sum_{j=1}^{N} A_{ij} \sin\Big( \theta_j(t - \tau_{ij}) - \theta_i(t) \Big)\,.
\label{eq:original_km_delay}
\end{equation}
The delay operator approach we introduce here generalizes to arbitrary adjacency matrices. In order to demonstrate this approach, we start by considering an undirected ring graph $\mathfrak{G}_{RG}$, where $N=100$ nodes are arranged on a one-dimensional ring with periodic boundary conditions. Each node in $\mathfrak{G}_{RG}$ is connected to the $k = 25$ nearest neighbors in each direction, and  $A_{ij} \in \{0,1\}$ is $1$ if oscillators $i$ and $j$ are connected, and $0$ otherwise. The time delay $\tau_{ij} = d_{ij} / \nu$ between two nodes $i$ and $j$ grows linearly with distance ($d_{ij}$) with respect to the periodic boundary conditions on the ring ($d_{ij} = \min( |i-j|, N - |i-j| )$). For the parameters chosen in this work, the time delays range from approximately $2$ to $62$ ms, a timescale relevant to neural dynamics \cite{cabral2011role,choi2019synchronization, petkoski2022normalizing}. We consider the case where all oscillators have the same frequency of $10$ Hz ($\omega = 20\pi$)\revision{; however,} our approach can be applied to the case of non-identical natural frequencies \revision{\cite{fig_s1}}.

The time-delay term $\theta_{j}(t-\tau_{ij})$ can be approximated by $\theta_{j}(t)-\omega \tau_{ij}$ \cite{ko2007effects,jeong2002time,breakspear2010generative}. Using this approximation, in combination with the algebraic approach to the Kuramoto dynamics \cite{muller2021algebraic,budzinski2022geometry}, we introduce a delay operator, which provides analytical insight into how heterogeneous time delays can create specific, sophisticated spatiotemporal structures in the resulting nonlinear dynamics. Applying this approximation to Eq.~(\ref{eq:original_km_delay}), we arrive at an equation that captures the time-delay dynamics in the dKM in heterogeneous phase lags \cite{ko2007effects,jeong2002time,breakspear2010generative}. \revision{We can then} use our algebraic approach to the Kuramoto dynamics and arrive at \revision{(see Appendix --  Sec. \ref{sec:complex_valued_approach} -- for details)}
\begin{equation}
\bm{x}(t) = e^{\i \omega t} e^{t \bm{W}} \bm{x}(0)\,,
\label{eq:analytical_solution}
\end{equation}
where $\bm{x} \in \mathbb{C}^{N}$, and the matrix $\bm{W}$ is given by
\begin{equation}
\bm{W} = \epsilon e^{-\i\bm{\eta}} \circ \bm{A}\,,
\label{eq:matrix_delay}
\end{equation}
where $\circ$ represents the Hadamard (elementwise) product. This matrix has information about the coupling strength $\epsilon$, the time delays $\bm{\eta} = \omega \bm{\tau}$ present in the original dKM, and the connection scheme of the system on $\bm{A}$. In previous work, we have shown this complex-valued equation, when evaluated through the procedure described below, precisely captures the trajectories of the original, nonlinear Kuramoto model \cite{budzinski2022geometry}. We now show that this approach generalizes to the case of heterogeneous time delays.

\revision{With this approach, we have} two dynamical systems:~the original, nonlinear KM and a complex-valued system with the explicit solution in Eq.~(\ref{eq:analytical_solution}) \revision{(details on the derivation can be found in \cite{budzinski2022geometry} and in the Appendix - Sec. \ref{sec:complex_valued_approach})}. In the complex-valued system, $\bm{x} \in \mathbb{C}^N$ has elements $x_i(t) \in \mathbb{C}$ whose argument we compare with the numerical solution of the original Kuramoto model with heterogeneous time delays (dKM) $\theta_i(t) \in \mathbb{R}$ (obtained by Euler integration of Eq.~(\ref{eq:original_km_delay}) with high temporal precision). \revision{That is,} $\mathrm{Arg}[x_{i}(t)]$ is compared with $\theta_{i}(t)$. When initialized with unit-modulus initial conditions $|x_i(0)| = 1~\text{for all}~i$, with arguments $\mathrm{Arg}[x_i(0)]$ that match the initial phases $\theta_i(0)$ in the original dKM, the trajectories in the original and complex-valued KM correspond for a non-trivial window of time \cite{muller2021algebraic}. \revision{As mentioned above,} in \cite{budzinski2022geometry} we found that iterating the explicit expression Eq.~(\ref{eq:analytical_solution}) in a specific manner produces trajectories in the complex-valued system that \textit{precisely} match those in the original, nonlinear Kuramoto model. Specifically, we can evaluate:
\begin{equation}
\bm{x}(t+\varsigma) = \Lambda \big[e^{\i \omega \varsigma} e^{\varsigma \bm{W}} \bm{x}(t)\big]\,,
\label{eq:lambda_operator_x}
\end{equation}
where $\varsigma$ is small but finite, $t \in [0, \varsigma, 2\varsigma, \cdots, n\varsigma]$, and $\Lambda$ represents an elementwise operator mapping the modulus of each state vector element $x_i(t)$ to unity. This approach represents an iterative analytical procedure, defined by the application of the linear matrix exponential and $\Lambda$. Note that Eq.~(\ref{eq:lambda_operator_x}) propagates the solution at discrete time intervals defined by $\varsigma$, Eq.~(\ref{eq:analytical_solution}) can be applied within intervals defined by $\varsigma$, and $\varsigma > dt$. Critically, while this iterative procedure does not represent a closed-form, all-time solution for the dynamics of the original nonlinear Kuramoto system, all evolution of the arguments $\mathrm{Arg}[x_i]$ (which, again, correspond with $\theta_i(t)~\forall~i$ in the original KM) is governed under the linear matrix exponential operator, and it is clear that the elementwise $\Lambda$ operator only changes the moduli. In this work, we show that this approach applies also in the case of heterogeneous time delays and provides analytical insight into how distance-dependent time delays create specific spatiotemporal patterns.

\section{Results}

We first study phase synchronization in networks with (dKM) and without (KM) time-delays on $\mathfrak{G}_{RG}$, as a function of the coupling strength $\epsilon$ \revision{(Fig. \ref{fig:kuramoto_order_parameter_eps})}. We use the Kuramoto order parameter:
\revision{
\begin{equation}
    R(t) = \frac{1}{N}\left|\sum_{j=1}^N e^{\i\theta_j(t)}\right|,
\end{equation}}
%
and its time average $\langle R \rangle$ for 10-second simulations \revision{to measure the level of phase synchronization}. As \revision{the coupling strength} $\epsilon$ increases in the non-delayed case (original KM and complex-valued KM), $\langle R \rangle$ begins at a low value and increases until approaching unity (representing phase synchronization).
\begin{figure}[hbt]
    \centering
    \includegraphics[width=0.95\columnwidth]{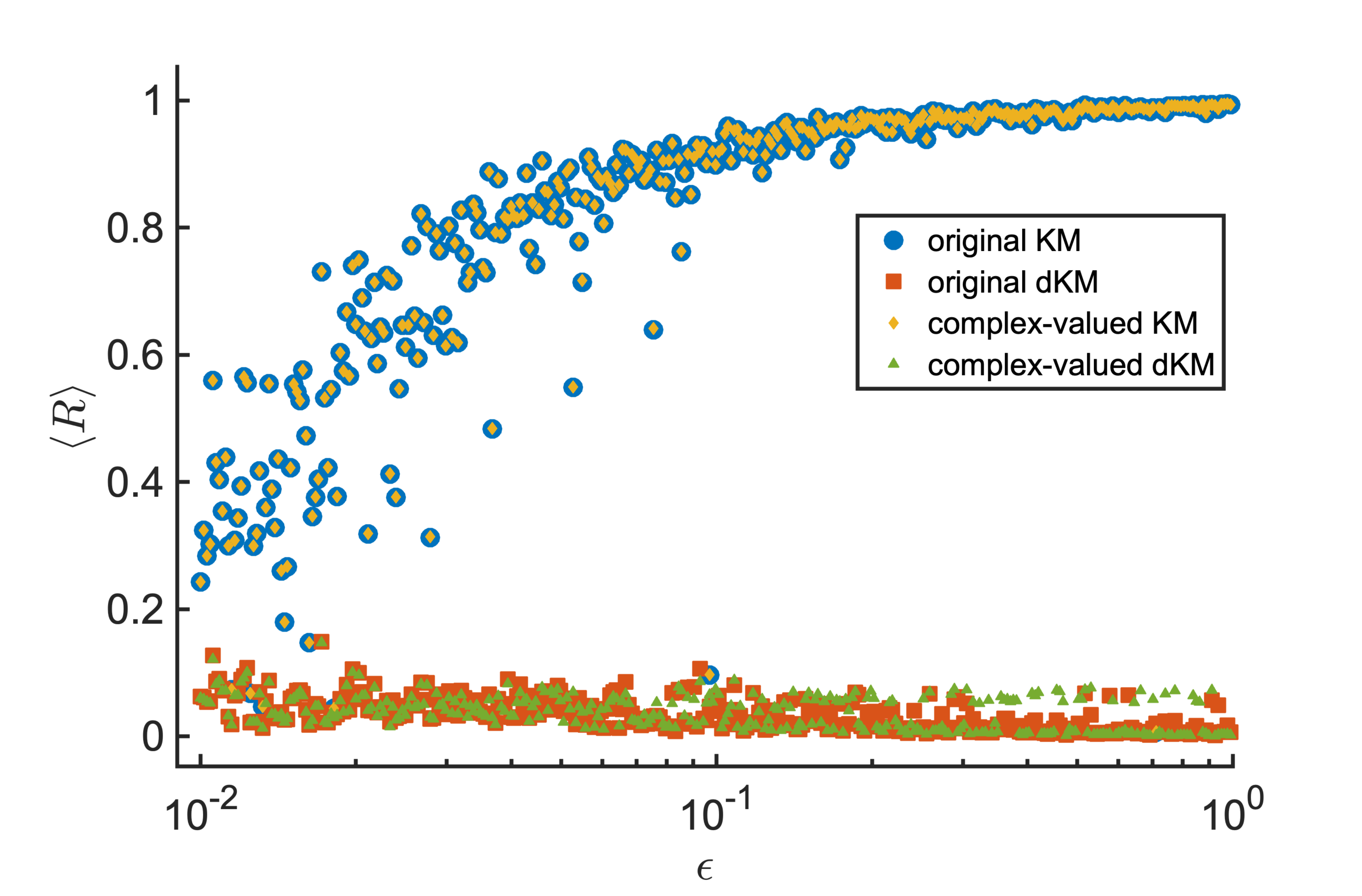}
    \caption{\textbf{Synchronization level for non-delayed and delayed networks.} The time-average Kuramoto order parameter $\langle R \rangle$ is plotted as a function of the coupling strength $\epsilon$ for the non-delayed case (blue dots: original KM, orange diamonds: complex-valued KM) and for the delayed case (red squares: original dKM, green triangles: complex-valued dKM). Each dot represents one 10-second simulation with random initial conditions ($\mathcal{U}(-\pi,\pi)$), which are the same for the complex-valued case and for the numerical simulation at each point.}
    \label{fig:kuramoto_order_parameter_eps}
\end{figure}
\begin{figure*}[htb]
    \centering
    \includegraphics[width=1.0\textwidth]{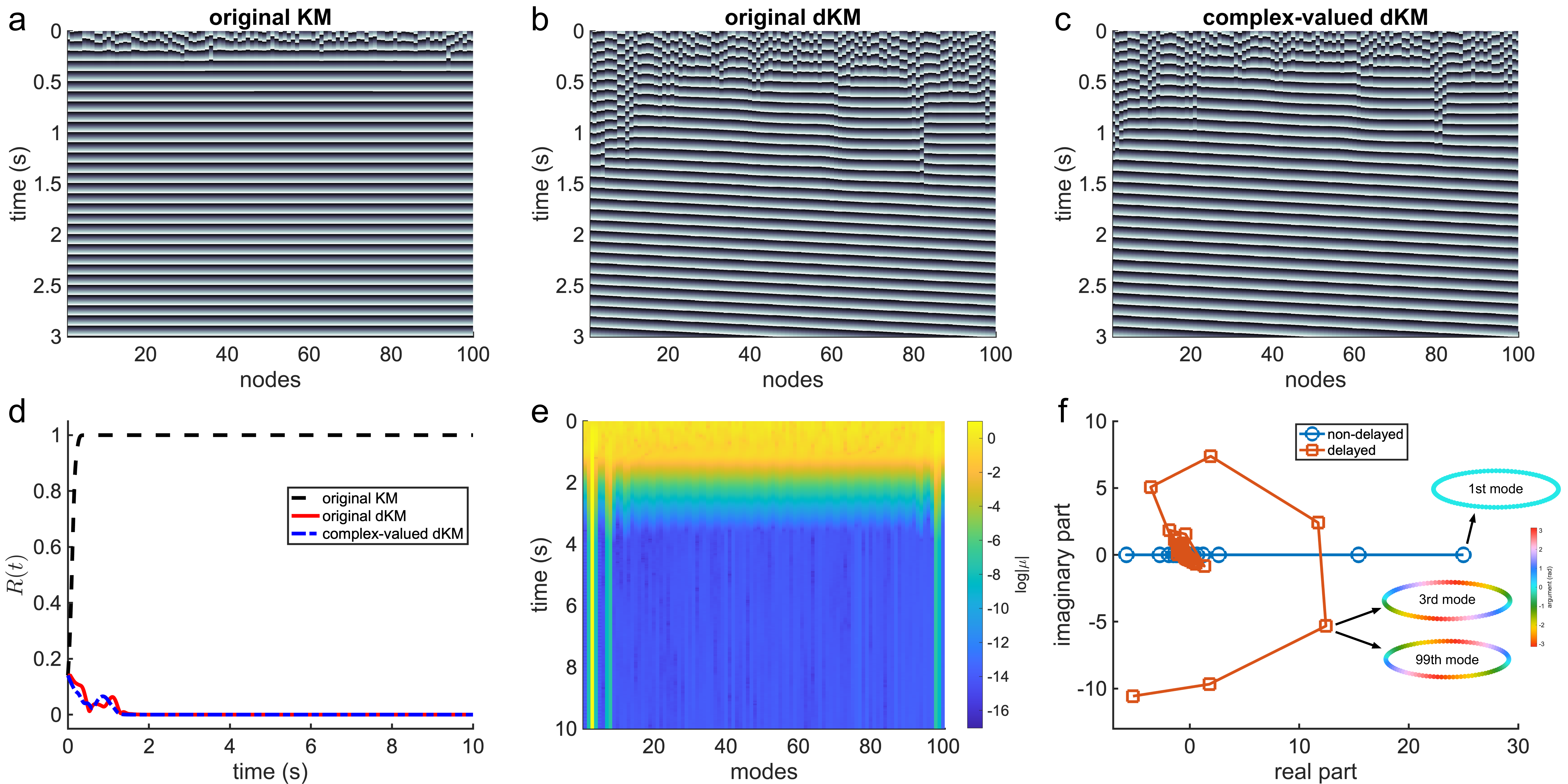}
    \caption{\textbf{Analytical and geometric view on the effect of time delays.} The spatiotemporal dynamics of the system is represented in color-code, where the phase of each oscillator is plotted as a function of time for \textbf{(a)} the original KM, \textbf{(b)} the original dKM, \textbf{(c)} and the complex-valued dKM. Dark colors represent phases close to $-\pi$ and light colors phases close to $\pi$. \textbf{(a)} Without delay, the network transitions to phase synchronization, which is represented by the horizontal lines. The effect of the delay, however, induces wave patterns in the system, whose dynamics are represented \textbf{(b)} in the original dKM and also captured \textbf{(c)} by the complex-valued model. \textbf{(d)} These dynamical characteristics are corroborated by the Kuramoto order parameter $R(t)$. \textbf{(e)} The eigenmodes offer a geometric perspective to such dynamics, where the waves are represented by a single eigenmode contribution ($3^{\mathrm{rd}}$ mode in this case). \textbf{(f)} The eigenvalues of $\bm{W}$ (delayed) and $\epsilon \bm{A}$ (non-delayed) provide further analytical insights into the effect of the delay in the system: it rotates the eigenvalues in the complex plane, which allows the system to access different modes. In the non-delayed case, the leading eigenvalue (in real part) is associated with an eigenvector with a zero phase difference configuration ($1^{\mathrm{st}}$ mode). In the delayed case, otherwise, there are two leading eigenvalues that are associated to the \revision{eigenvectors $\bm{v}_3$ and $\bm{v}_{99}$, which have phase configurations representing traveling waves}.}
    \label{fig:general_scenario}
\end{figure*}

In the case with heterogeneous time delay (original dKM and complex-valued dKM), the order parameter remains low (red squares and green triangles), reflecting the fact that time delays induce a range of spatiotemporal patterns, as observed previously \cite{yeung1999time, jeong2002time, ko2007effects, laing2016travelling, timms2014synchronization, ermentrout2009delays, ko2004wave}. Here, we observe that the complex-valued model is able to capture the average dynamics that the original Kuramoto model depicts, for both the non-delayed and delayed cases, for different coupling strength across different initial conditions \revision{(Fig. \ref{fig:kuramoto_order_parameter_eps})}.

We next study dynamics in the KM and dKM considering an individual realization, for a fixed coupling strength ($\epsilon = 0.5$), and compare the dynamics of the original dKM to the evaluation of the complex-valued approach. Without time delays, the original KM exhibits a quick transition from random initial conditions to a phase synchronized state (horizontal lines, Fig.~\ref{fig:general_scenario}a). With time delays, however, phase synchronization is not reached, and the original dKM exhibits a transition from random initial conditions to a traveling wave state (diagonal structures, Fig.~\ref{fig:general_scenario}b). The evaluation of the complex-valued dKM captures both the transient dynamics and traveling wave state exhibited in the original dKM (Fig.~\ref{fig:general_scenario}c), as well as the dynamics of the Kuramoto order parameter $R(t)$ (Fig.~\ref{fig:general_scenario}d).

Our approach to systems with heterogeneous time delay provides insight into the mechanism for these dynamics in terms of the spectrum of $\bm{W}$ \revision{-- Eq. (\ref{eq:matrix_delay})}. If $\bm{A}$ and $\bm{\tau}$ are circulant, $\bm{W}$ is also circulant \revision{(see Appendix -- Sec. \ref{sec:hadamard_product})}, hence $\bm{W}$ and $\bm{A}$ share the same eigenvectors (which form an orthonormal basis). We can then write Eq.~(\ref{eq:analytical_solution}) using the eigenspectrum of $\bm{W}$, which results in $\bm{x}(t) = e^{\i \omega t} ( \alpha_{1}e^{\lambda_{1}t}\bm{v}_{1} + \cdots + \alpha_{N}e^{\lambda_{N}t}\bm{v}_{N} )$, where $\alpha_{i}$ can be written in terms of initial conditions. Importantly, we can also write Eq.~(\ref{eq:lambda_operator_x}) in a similar fashion, which results in $\bm{x}(t+\varsigma) = \Lambda \big[ e^{\i \omega \varsigma} ( \alpha_{1}e^{\lambda_{1}\varsigma}\bm{v}_{1} + \cdots + \alpha_{N}e^{\lambda_{N}\varsigma}\bm{v}_{N} ) \big]$, where $\alpha_i$ can again be written in terms of the state of the system at time $t \in [0, \varsigma, 2\varsigma, \cdots, n\varsigma]$. Thus, while it is in general a very difficult problem to understand the dynamics of nonlinear networks in terms of eigenspectra, this approach provides a unique insight into the connection between the spectrum of $\bm{W}$ \revision{-- Eq.~(\ref{eq:matrix_delay}) --} and the spatiotemporal dynamics of the nonlinear oscillator network \revision{-- Eq. (\ref{eq:original_km_delay})}. Critically, our approach uses familiar mathematical techniques from linear algebra matrix theory in a distinct way:~while previous approaches in nonlinear dynamics have sought to describe the dynamics using the spectrum of the Laplacian matrix \cite{grabow2010small,sorrentino2016complete,sugitani2021synchronizing}, the focus on the complex-valued system in our approach enables the insight that the \textit{argument} of the eigenvectors of the matrix $\bm{W}$ provides analytical predictions about the resulting nonlinear dynamics.

Following this idea, Fig. \ref{fig:general_scenario}e shows the eigenmode contributions, here represented by $\log{|\mu_i|}$, as a function of time, for the dynamics in Fig.~\ref{fig:general_scenario}c. Here, the eigenmode contributions are given by the projection of the complex-valued approach solution $\bm{x}(t)$ onto the eigenvectors of $\bm{W}$. \revision{The eigenmode contributions are obtained as} $\mu_{k}(t) = \langle \bm{x}(t), \bm{v}_{k} \rangle$, where $\langle . \rangle$ denotes the standard \revision{complex} inner product. Figure \ref{fig:general_scenario}e shows that, when the network exhibits incoherent dynamics, the eigenmode contributions remain uniform across $\mu_i$. When the traveling wave pattern is reached, on the other hand, the $3^{\mathrm{rd}}$ eigenmode becomes dominant (note the log scale). These results demonstrate that the change from incoherent dynamics to a traveling wave can be understood quite directly through the geometry of the eigenmodes. Further, in the case of circulant networks, we can evaluate eigenvalues and eigenvectors analytically using the circulant diagonalization theorem \revision{(CDT) \cite{davis1979}}\revision{, and in this case,} the $1^{\mathrm{st}}$ eigenvector represents the solution where all oscillators have the same phase (phase synchronization), and higher modes represent wave patterns, given by Fourier modes \revision{(see Appendix  - Sec. \ref{sec:cdt})}.

The effect of heterogeneous time delays on the dynamics of the dKM can be understood through the geometry of eigenvalues in the complex plane. Figure \ref{fig:general_scenario}f illustrates the eigenvalues of $\epsilon \bm{A}$ (non-delayed) and $\bm{W}$ (delayed). While the non-delayed case (blue line and dots) has purely real eigenvalues, the effect of the heterogeneous time delays (red line and squares) can be understood in our framework in terms of the Hadamard (elementwise) product of the delay operator $\bm{\tau}$ and $\bm{A}$ (see Eq.~(\ref{eq:matrix_delay}), and \revision{Appendix - Sec. \ref{sec:hadamard_product}}). The effect of this operation is to provide a specific rotation of the eigenvalues in the complex plane. This rotation allows the system to access higher modes and, therefore, to exhibit different traveling wave patterns. Further, the rotation is not the same for all eigenvalues because the delays are heterogeneous. In this particular case, the rotation leads to eigenvalues associated to the $3^{\mathrm{rd}}$ and $99^{\mathrm{th}}$ modes to have the largest real part, \revision{allowing} the system \revision{to} reach traveling wave \revision{states} associated with the $3^{\mathrm{rd}}$ and $99^{\mathrm{th}}$ modes. In the particular example of Fig. \ref{fig:general_scenario}, the network evolves to a wave given by the $3^{\mathrm{rd}}$ mode, but different (random) initial conditions can either evolve to the dynamics described by the $3^{\mathrm{rd}}$ or $99^{\mathrm{th}}$ mode \revision{\cite{fig_s2}}. Moreover, when different time delays are considered, different modes can be dominant, and therefore the system evolves to a different wave pattern \revision{\cite{fig_s3}}.

\begin{figure}[htb]
    \centering
    \includegraphics[width = \columnwidth]{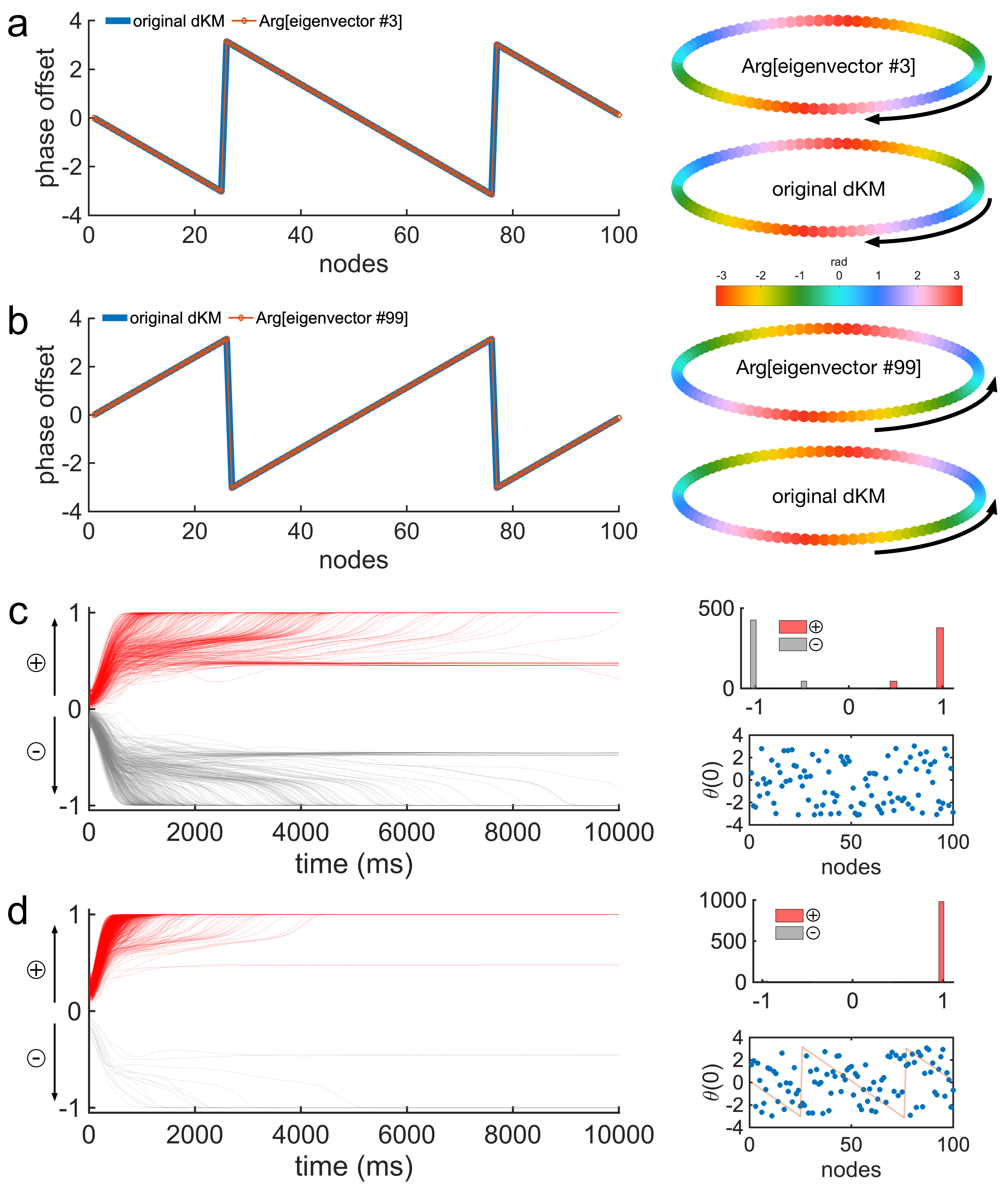}
    \caption{\textbf{Analytical predictions of specific wave patterns.} \textbf{(a)} The phase configuration for the original dKM (blue line) matches the argument (elementwise) of the $3^{\mathrm{rd}}$ eigenvector $\mathrm{Arg}[\bm{v}_{3}]$ - analytical prediction - (red dotted line). A representation on the circle using a color-code reveals the wave pattern (right). \textbf{(b)} Different initial conditions lead to the wave pattern that matches the argument of the $99^{\mathrm{th}}$ eigenvector $\mathrm{Arg}[\bm{v}_{99}]$. These waves can propagate either counterclockwise (negative) or clockwise (positive). \textbf{(c)} With random initial conditions, due to the dominance of two eigenvalues ($3^{\mathrm{rd}}$ and $99^{\mathrm{th}}$), the system exhibits waves propagating in both directions - with approximately half of the initial conditions evolving to each direction (top right). \textbf{(d)} With biased initial conditions, starting from ${\mathrm{Arg}[\bm{v}_{3}]}$ (red line, bottom right) and adding uniform random phases $0.8(\mathcal{U}(-\pi,\pi))$, we obtain a preferred direction of propagation.}
    \label{fig:eigenvector_dynamics}
\end{figure}
We can now uncover how the combination of network structure, time delays, and node state can create specific spatiotemporal patterns. By using our delay operator approach, we can analytically predict the specific pattern to which the original dKM evolves. Figure \ref{fig:eigenvector_dynamics}a shows the wave pattern given by $\theta$ obtained from the original dKM (blue line) and the argument (elementwise) of the $3^{\mathrm{rd}}$ eigenvector (red line), which predicts the observed dynamics \cite{movie_s1}. In this case, phases increase in the clockwise direction around the ring, which we define to be the positive direction ($+1$). It is important to note that, in our approach, the argument of each eigenvector element $(\mathrm{Arg}[(\bm{v}_k)_i] \, \forall i \in [1,N])$ directly relates with the phase offset in the resulting network dynamics. Because of the correspondence between trajectories in the complex-valued model and the original dKM, this approach creates a direct link between eigenvectors of the adjacency matrix and the specific spatiotemporal dynamics that result. For the dynamics in Fig.\,\ref{fig:eigenvector_dynamics}a, the eigenmode contribution is given by $\mu_{3}$ (see Fig. \ref{fig:general_scenario}e), and the phase configuration matches the argument of $\bm{v}_3$. In the example considered here, two eigenvalues are dominant (i.e.\,having the largest real part) - $\lambda_3$ and $\lambda_{99}$ (Fig.\,\ref{fig:general_scenario}f). Different initial conditions can thus evolve to the phase pattern given by the $99^{\mathrm{th}}$ mode, which is predicted by $\bm{v}_{99}$ (Fig.\,\ref{fig:eigenvector_dynamics}b). In this case, the spatial frequency is the same as observed in the previous case, but the direction of the wave pattern is the opposite \cite{movie_s2}. These results show a clear connection between the spectrum of the network (described by $\bm{W}$) and the dynamics on the original dKM, where the wave pattern (solution) can be described by the phase configuration of the eigenvector associated to the dominant mode.

We take counterclockwise increases in phase to be in the negative direction, and clockwise increases to be positive. Because the network considered here has two dominant eigenvalues equal in their real parts, random initial conditions evolve equally either to the phase pattern of $\bm{v}_3$ or $\bm{v}_{99}$ in individual simulations (Fig.\,\ref{fig:eigenvector_dynamics}c). \revision{To quantify the spatiotemporal dynamics, the spatial frequency, and the direction of propagation, we compare the phases obtained from the original dKM and the argument of the eigenvectors of $\bm{W}$. Specifically, we evaluate:
 \begin{equation}
     \rho^{(k)}(t) = \left| \frac{1}{N} \sum_{i}^{N} e^{\i \theta_{i}(t)} e^{-\i \mathrm{Arg}[(\bm{v}_{k})_{i}]} \right|,
\end{equation}
where $\theta_{i}(t)$ is the phase of the oscillators $i$ at time $t$ obtained from the original dKM, $N$ is the number of oscillators in the network, $\i$ is the imaginary unit, and $\bm{v}_{k}$ is the $k^{\mathrm{th}}$ eigenvector of $\bm{W}$. Here, we use $\bm{v}_{3}$, and $\rho^{(k)} = 1$ means the phase configuration of the network given by the $\bm{\theta}(t)$ is the same as the one given by the argument of the eigenvector $\bm{v}_{k}$.} In the case shown in Fig. \ref{fig:eigenvector_dynamics}c, approximately half of the simulations evolve to the positive direction\revision{, indicating the dynamics matches the argument of $\bm{v}_{3}$}, and approximately half evolve to the negative\revision{, indicating the dynamics is given by the argument of $\bm{v}_{99}$}. A small fraction of initial conditions exhibit inner products approximately $\pm 0.5$, corresponding to a wave with a different spatial frequency.

Using the insights from this approach, we can now design initial conditions that generate waves in a preferred direction. To do this, we started from the phase pattern specified by $\bm{v}_3$ and randomized the phases by nearly a full cycle ($0.8~\mathcal{U}[-\pi,\pi]$, then wrapped in $[-\pi,\,\pi]$). While this initial condition is nearly random (Fig.\,\ref{fig:eigenvector_dynamics}d, bottom right, where the red line represents ${\mathrm{Arg}[\bm{v}_{3}]}$; compare with Fig.\,\ref{fig:eigenvector_dynamics}c, bottom right), nearly all simulations evolve to the positive direction. These results demonstrate that the combination of connectivity, time delays, and network state can generate specific spatiotemporal patterns in oscillator networks -- here, traveling waves with a chirality in a preferred direction.

\revision{The framework for systems with heterogeneous time delays introduced in the work generalizes to many types of networks. This approach can be applied to very sophisticated networks obtained from experimental data. In particular, this approach can successfully predict traveling wave patterns arising in an oscillator network based on connectivity in the human brain. Figure \ref{fig:hcp_network} illustrates simulations and the analytical prediction resulting from our approach for networks where the connectivity data is based on the Human Connectome Project (HCP) \cite{hagmann2008mapping}. In this case, $N = 998$ cortical regions are given at a point in 3-space, with connections between areas derived from neuroimaging data. Connection weights between regions are determined by the number of fibers \cite{hagmann2008mapping,muller2014brain}, which we use to build the adjacency matrix $\bm{A}$. Here, the coupling strength is scaled with $\epsilon = 200$, and the initial conditions for each analysis \revision{are} given by random phases $[-\pi, \pi]$. Further, time delays are \revision{obtained} by $\tau_{ij} = d_{ij} / \nu$, where the distances $d_{ij}$ are determined by the average length of these fibers, and the known axonal conduction speed is given by $\nu = 5$ m/s \cite{swadlow2012axonal}. The dynamics of each node is represented by the Kuramoto model, given either by Eq. (\ref{eq:original_km_no_delay}) in the nondelayed case and by Eq. (\ref{eq:original_km_delay}) in the delayed case. The natural frequency of each oscillator is given by $10$ Hz (simulating, for example, a specific drive from the thalamus). \revision{Using the} delay operator, we construct the matrix $\bm{W}$ for these systems -- Eq. (\ref{eq:matrix_delay}) -- which allows us to obtain analytical predictions of the spatiotemporal patterns that emerge.  
\begin{figure*}[htb]
    \centering
    \includegraphics[width=0.95\textwidth]{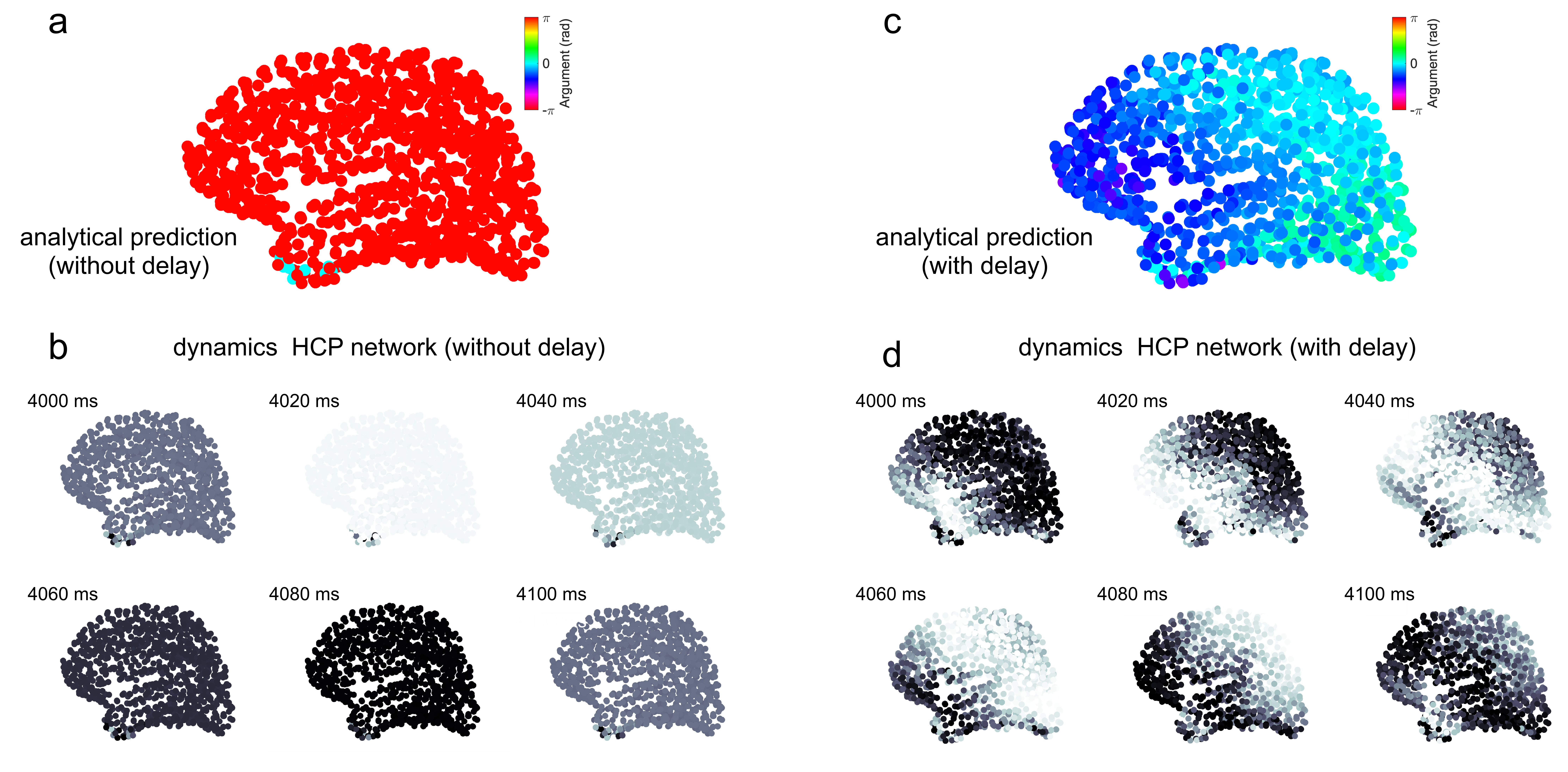}
    \caption{\revision{\textbf{Analytical predictions of spatiotemporal patterns in brain networks.} We use our approach to investigate networks based on the \textit{Human Connectome Project} (HCP) \cite{hagmann2008mapping}. We consider the case without delay in the coupling between nodes and also the case with distance dependent delays (heterogeneous delay). We use our delay operator to create the matrix $\bm{W}$, which allows analytical predictions of the dynamics. We show the phase of each node, given by the Kuramoto model, using a color-code (dark colors are values close to $-\pi$, and light colors are values close to $\pi$). \textbf{(a)} In the case without delay, the argument of the leading eigenvector depicts zero phase difference, which predicts phase synchronization. \textbf{(b)} We then study the numerical simulation for the network without delay, given by Eq. (\ref{eq:original_km_no_delay}), which shows a phase synchronized state. \textbf{(c)} In the case with heterogeneous time delays, the argument of the leading eigenvector shows a phase offset from bottom left to the top right (in the projection), which predicts a wave pattern. \textbf{(d)} We perform the numerical simulation with the delayed Kuramoto model, given by Eq. (\ref{eq:original_km_delay}), and the network depicts the wave pattern that is predicted by our approach. This shows that we are able to predict the dynamics observed in the simulations using our delay operator.}}
    \label{fig:hcp_network}
\end{figure*}

First, we consider the case without time delays, where $\tau_{ij} = 0$. We then obtain the eigenspectrum of the matrix $\bm{W}$ and plot the argument (elementwise) of the eigenvector associated with the leading eigenvalue (Fig. \ref{fig:hcp_network}a). In this case, this eigenvector shows a zero phase difference across nodes, \revision{predicting} phase synchronization. We then perform the numerical simulation of the Kuramoto model (without delay), given by Eq. (\ref{eq:original_km_no_delay}) and plot the phase of each node in color-code (Fig. \ref{fig:hcp_network}b), where we observe a phase synchronized behavior \cite{movie_s3}.

On the other hand, when we consider time delays in the interaction between cortical areas, the scenario is different. In this case, the argument (elementwise) of the eigenvector associated with the leading eigenvalue depicts a phase offset increasing from the bottom left to top right (in this projection), \revision{predicting} a wave propagating along that direction (Fig. \ref{fig:hcp_network}c). We then perform numerical simulations of the Kuramoto model with heterogeneous time delays -- Eq. (\ref{eq:original_km_delay}) and we observe the wave pattern that is predicted by our approach, as shown in Fig. \ref{fig:hcp_network}d \cite{movie_s4}.

This example now clearly demonstrates the advantage of this analytical approach:~when we numerically evaluate the eigenspectrum of $\bm{W}$ in this case, the leading eigenvector for the case without delays predicts phase synchrony, while the leading eigenvector for the case with delays predicts the {\it precise} wave pattern observed in the simulation. This result shows that our approach is able to predict the spatiotemporal pattern that results from connectivity and time delays in a highly relevant, real-world case.
}

\section{Conclusion}

In this work, we have introduced an analytical approach to the dynamics of nonlinear oscillator networks with heterogeneous time delays, an important open problem in physics with many potential applications. The advance in this work is based on an algebraic approach to the Kuramoto model introduced in \cite{budzinski2022geometry}. Importantly, the flexibility of this framework allowed us to introduce a delay operator, which provides rigorous analytical predictions for the specific traveling wave patterns induced by distance-dependent time delays. Using this approach, we can explain the effect of time delays in terms of a rotation of the eigenvalues of the matrix describing the system, which provides a clear and precise way to understand heterogeneous time delays in terms of the geometry of eigenmodes. Our approach therefore allows analytical predictions for the specific spatiotemporal patterns exhibited by the original dKM. 

\revision{This} framework allows us to understand how the combination of isotropic connectivity and time delays can produce traveling waves propagating in a preferred direction, as observed in experimental data \cite{muller2016rotating}. Importantly, while this question first arose in our study of neural dynamics in human cortex during sleep, the approach we have introduced here is general to networks of oscillators at finite scales. The results shown in this work, together with the results in \cite{budzinski2022geometry,muller2021algebraic}, represent a coherent and general framework for nonlinear oscillator networks. 

The central advance of this framework is to consider the dynamics in an individual simulation, taking into account both the initial conditions and the specific connectivity pattern in the network. This framework thus provides an opportunity to connect an individual adjacency matrix, for example a single network taken from experimental data or a single realization of a random graph model, to the specific spatiotemporal pattern that results in a simulation. This approach has important potential applications, for example in linking an experimentally reconstructed brain network to dynamics and computation in a neural system \cite{litvina2019brain} or in linking the connections in a large-scale power grid to potential large and transient disruptions \cite{witthaut2022collective,motter2013spontaneous}. In this work, we have generalized this framework to systems with heterogeneous time delays, which demonstrates the utility of this algebraic, operator-based approach to nonlinear dynamical systems at finite scales.

\begin{acknowledgments}
This work was supported by BrainsCAN at Western University through the Canada First Research Excellence Fund (CFREF), the NSF through a NeuroNex award (\#2015276), the Natural Sciences and Engineering Research Council of Canada (NSERC) grant R0370A01, ONR N00014-16-1-2829, NIH EB009282, NIH EB026899, the Swartz Foundation, SPIRITS 2020 of Kyoto University, Compute Ontario (computeontario.ca), Compute Canada (computecanada.ca), and the Western Academy for Advanced Research. J.M.~gratefully acknowledges the Western University Faculty of Science Distinguished Professorship in 2020-2021. R.C.B gratefully acknowledges the Western Institute for Neuroscience Clinical Research Postdoctoral Fellowship. G.B.B gratefully acknowledges the Canadian Open Neuroscience Platform (Graduate Scholarship), the Vector Institute (Postgraduate Affiliate), and NSERC (Canada Graduate Scholarship - Doctoral).
\end{acknowledgments}

\section*{Code availability}

An open-source code repository for this work is available on GitHub: \href{http://mullerlab.github.io}{\textcolor{Cerulean}{http://mullerlab.github.io}}.

\revision{
\section{Appendix}

\subsection{The complex-valued approach}\label{sec:complex_valued_approach}

We consider the Kuramoto model with heterogeneous time delays described by Eq. (\ref{eq:original_km_delay}) and then use the approximation given by  $\theta_{j}(t - \tau_{ij}) \approx \theta_{j}(t) - \omega \tau_{ij}$ \cite{ko2007effects,jeong2002time,breakspear2010generative}, which leads to
\begin{equation}
\dot{\theta}_i(t) = \omega + \epsilon \sum_{j=1}^{N} A_{ij} \sin \Big( \theta_j(t) - \theta_i(t) - \eta_{ij} \Big),
\label{eq:original_km_phase_lag}
\end{equation}
where $\eta_{ij} = \omega \tau_{ij}$.

Based on \cite{muller2021algebraic, budzinski2022geometry}, we introduce the complex-valued approach to the Kuramoto model described by Eq. (\ref{eq:original_km_phase_lag}). To do that, we introduce a new dynamical system, described by the variable $\psi \in \mathbb{C}$: 
\begin{equation}
\begin{split}
\dot{\psi}_i(t) = \omega + \epsilon \sum_{j=1}^{N} A_{ij} \Big[ \sin\Big( \psi_j(t) - \psi_i(t) - \eta_{ij}\Big) \\
- \i \cos\Big( \psi_j(t) - \psi_i(t) -\eta_{ij}\Big) \Big]\,.
\end{split}
\end{equation}
Next, multiplying both sides by $\i$ and applying Euler's formula yields
\begin{equation}
\i\dot{\psi}_i(t) = \i\omega + \epsilon e^{-\i\psi_i(t)} \sum_{j=1}^{N} A_{ij} e^{\i\psi_j(t)} e^{-\i \eta_{ij}}.
\label{eq:complex_valued_1}
\end{equation}
We define $\bm{W}$ as:
\begin{equation}
\bm{W} = \epsilon e^{-\i\bm{\eta}} \circ \bm{A}\,,
\end{equation}
where $\circ$ represents the Hadamard product (or element-wise product), and $\eta_{ij} = \omega \tau_{ij}$. This results in the following matrix form of Eq. (\ref{eq:complex_valued_1}):
\begin{equation}
\bm{\dot{\psi}}(t) = \bm{\omega} + \frac{1}{\i}\,\text{diag}[ e^{-\i\bm{\psi}(t)} ]\,\bm{W} e^{\i\bm{\psi}(t)}\,,
\end{equation}
where we note explicitly that $\bm{\psi}=[\psi_1,\cdots,\psi_N]^T$, $\bm{\dot{\psi}}=[\dot{\psi}_1,\cdots,\dot{\psi}_N]^T$, and $\bm{\omega}=[\omega,\cdots,\omega]^T$. Furthermore, we can write the previous equation as:
\begin{equation}
\frac{d}{dt} e^{\i\bm{\psi}(t)} = \bigg( \text{diag}[ \i \bm{\omega} ] + \bm{W} \bigg) e^{\i\bm{\psi}(t)}\,.
\end{equation}
Lastly, letting $\bm{x}(t) = e^{\i\bm{\psi}(t)}$, we have
\begin{equation}\label{eqSolSupp}
\bm{\dot{x}}(t) = \bigg( \text{diag}[ \i \bm{\omega} ] + \bm{W} \bigg) \bm{x}(t)\,,
\end{equation}
whose general solution is
\begin{equation}
\bm{x}(t) = e^{\i \omega t} e^{t \bm{W}} \bm{x}(0)\,.
\end{equation}
In this work, the dynamics of the complex-valued approach is studied by considering the elementwise argument of $\bm{x}(t)$, i.e. $\mathrm{Arg}[x_{i}(t)]~\forall~i \in [1,N]$. As shown in \cite{budzinski2022geometry}, when $\frac{|x_j|}{|x_i|} \approx 1$, the dynamics of $\mathrm{Arg}[\bm{x}(t)]$ precisely matches the trajectories of the Kuramoto model given by Eq. (\ref{eq:original_km_phase_lag}). This allows us to use the eigenspectrum of $\bm{W}$ to understand and predict the dynamics of the Kuramoto model with heterogeneous time delays.

\subsection{Circulant networks and Hadamard product}\label{sec:hadamard_product}

The definition of the Hadamard product can be described as follows:
\begin{definition}
Let $\bm{A}, \bm{B}$ be two $n \times n$ matrices. The Hadamard product $\bm{A} \circ \bm{B}$ is a matrix of dimension $n \times n$ with elements given by 
\[ (\bm{A} \circ \bm{B})_{ij}= (\bm{A})_{ij} (\bm{B})_{ij} .\] 
\end{definition}

For a complex number $\lambda$, we also define $e^{\circ (\lambda \bm{A})}$ to be the matrix of dimension $n \times n$ with elements given by 
\[ (e^{\circ \lambda \bm{A}})_{ij}=e^{\lambda \bm{A}_{ij}} .\] 

We have the following observation. 
\begin{prop}\label{prop:circ}
Let $\bm{A}, \bm{B}$ be two circulant matrices. Then
\begin{enumerate}
    \item $\bm{A} \circ \bm{B}$ is a circulant matrix. 
    \item $e^{\lambda \circ \bm{A}}$ is a circulant matrix. 
    
\end{enumerate}
\end{prop}
\begin{proof}
Assume that $\bm{A}=\text{circ}(\bm{a}), \bm{B}=\text{circ}(\bm{b})$ with $\bm{a}=(a_1, a_2, \ldots, a_n)$ and $\bm{b}=(b_1, b_2, \ldots, b_n).$ Then we can see that 
\[ \bm{A} \circ \bm{B}= \text{circ}((a_1b_1, a_2b_2, \ldots, a_n b_n)) ,\]
and 
\[ e^{\circ \lambda \bm{A}}=\text{circ}((e^{\lambda a_1}, e^{\lambda a_2}, \ldots, e^{\lambda a_n})) .\] 
Therefore, we conclude that both $\bm{A} \circ \bm{B}$ and $e^{\lambda \circ \bm{A}}$ are circulant. 
\end{proof} 

\subsection{The Circulant Diagonalization Theorem}\label{sec:cdt}

In the case of circulant networks, we can use the circulant diagonalization theorem (CDT) to obtain the eigenspectrum of the adjacency matrix analytically \cite{davis1979}. In our work, both the non-delayed network $\epsilon \bm{A}$ and the delayed one $\bm{W}$ are circulant (see Prop. \ref{prop:circ}).

\begin{figure}[thb]
    \centering
    \includegraphics[width=0.9\columnwidth]{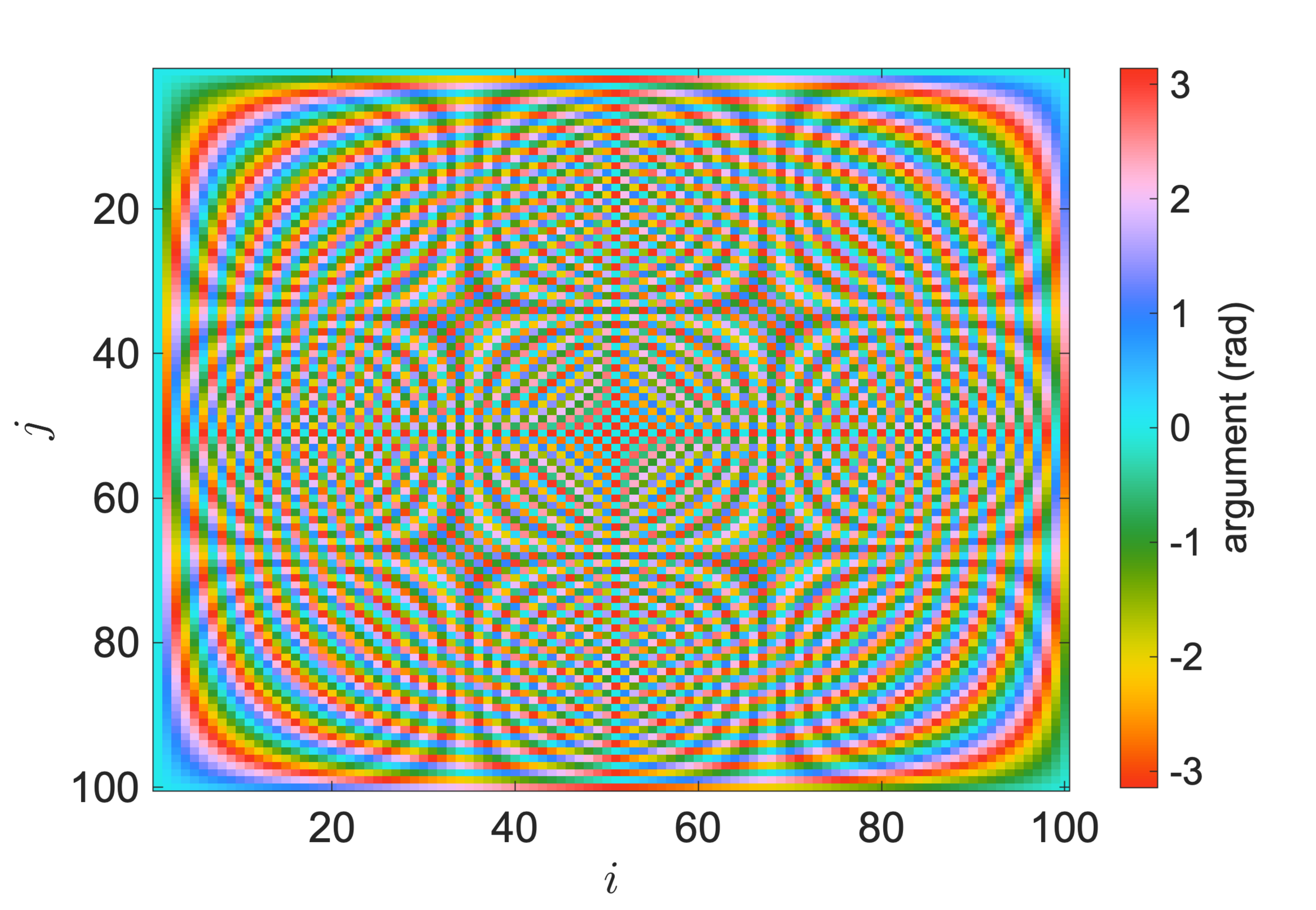}
    \caption{\revision{A graphical representation of the matrix with the phase configuration of the eigenvectors of $\bm{W}$. The $k^{\mathrm{th}}$ column is the color-coded argument (elementwise) of the $k^{\mathrm{th}}$ eigenvector.}}
    \label{fig:dft}
\end{figure}
The CDT states that all circulant matrices, say $\bm{H} = \text{circ}(\bm{h})$, where $\text{circ}(\bm{h})$ is the circulant matrix constructed from the generating vector $\bm{h}=(h_1, \cdots, h_N)$, are diagonalized by the same unitary matrix $\bm{U}$ with components
\begin{equation}
U_{ks} = \frac{1}{\sqrt{N}} \exp\left[ - \frac{2\pi \i}{N} (k-1)(s-1) \right] ,
\end{equation}
where $k,s \in [1,N]$ and that the $N$ eigenvalues are given by
\begin{equation}
E_k(\bm{H}) = \sum\limits_{j=1}^{N} h_j \exp\left[ - \frac{2\pi \i}{N} (k-1)(j-1) \right]\,.
\label{eq:eigenvalues_cdt}
\end{equation}
We let Eq. (\ref{eq:eigenvalues_cdt}) determine the ordering of the eigenvalues throughout this work. The argument of the eigenvectors associated with these eigenvalues correspond to the columns of the discrete Fourier transform (DFT) matrix, which range from low to high spatial frequencies.

Figure \ref{fig:dft} shows the argument of the eigenvectors in color-code. Here, $\mathrm{Arg}[(\bm{v}_1)_i] = 0~\forall~i \in [1,N]$ (as shown in Fig. \ref{fig:general_scenario}), which represents zero phase difference across oscillators, or phase synchronization. The other eigenvectors represent Fourier modes (waves) with different spatial frequencies. Figure \ref{fig:general_scenario} shows the cases of the eigenvectors $\bm{v}_{3}$ and $\bm{v}_{99}$.

} 


\begin{thebibliography}{49}%
\makeatletter
\providecommand \@ifxundefined [1]{%
 \@ifx{#1\undefined}
}%
\providecommand \@ifnum [1]{%
 \ifnum #1\expandafter \@firstoftwo
 \else \expandafter \@secondoftwo
 \fi
}%
\providecommand \@ifx [1]{%
 \ifx #1\expandafter \@firstoftwo
 \else \expandafter \@secondoftwo
 \fi
}%
\providecommand \natexlab [1]{#1}%
\providecommand \enquote  [1]{``#1''}%
\providecommand \bibnamefont  [1]{#1}%
\providecommand \bibfnamefont [1]{#1}%
\providecommand \citenamefont [1]{#1}%
\providecommand \href@noop [0]{\@secondoftwo}%
\providecommand \href [0]{\begingroup \@sanitize@url \@href}%
\providecommand \@href[1]{\@@startlink{#1}\@@href}%
\providecommand \@@href[1]{\endgroup#1\@@endlink}%
\providecommand \@sanitize@url [0]{\catcode `\\12\catcode `\$12\catcode
  `\&12\catcode `\#12\catcode `\^12\catcode `\_12\catcode `\%12\relax}%
\providecommand \@@startlink[1]{}%
\providecommand \@@endlink[0]{}%
\providecommand \url  [0]{\begingroup\@sanitize@url \@url }%
\providecommand \@url [1]{\endgroup\@href {#1}{\urlprefix }}%
\providecommand \urlprefix  [0]{URL }%
\providecommand \Eprint [0]{\href }%
\providecommand \doibase [0]{https://doi.org/}%
\providecommand \selectlanguage [0]{\@gobble}%
\providecommand \bibinfo  [0]{\@secondoftwo}%
\providecommand \bibfield  [0]{\@secondoftwo}%
\providecommand \translation [1]{[#1]}%
\providecommand \BibitemOpen [0]{}%
\providecommand \bibitemStop [0]{}%
\providecommand \bibitemNoStop [0]{.\EOS\space}%
\providecommand \EOS [0]{\spacefactor3000\relax}%
\providecommand \BibitemShut  [1]{\csname bibitem#1\endcsname}%
\let\auto@bib@innerbib\@empty
\bibitem [{\citenamefont {Muller}\ \emph {et~al.}(2016)\citenamefont {Muller},
  \citenamefont {Piantoni}, \citenamefont {Koller}, \citenamefont {Cash},
  \citenamefont {Halgren},\ and\ \citenamefont
  {Sejnowski}}]{muller2016rotating}%
  \BibitemOpen
  \bibfield  {author} {\bibinfo {author} {\bibfnamefont {L.}~\bibnamefont
  {Muller}}, \bibinfo {author} {\bibfnamefont {G.}~\bibnamefont {Piantoni}},
  \bibinfo {author} {\bibfnamefont {D.}~\bibnamefont {Koller}}, \bibinfo
  {author} {\bibfnamefont {S.~S.}\ \bibnamefont {Cash}}, \bibinfo {author}
  {\bibfnamefont {E.}~\bibnamefont {Halgren}},\ and\ \bibinfo {author}
  {\bibfnamefont {T.~J.}\ \bibnamefont {Sejnowski}},\ }\bibfield  {title}
  {\bibinfo {title} {Rotating waves during human sleep spindles organize global
  patterns of activity that repeat precisely through the night},\ }\href@noop
  {} {\bibfield  {journal} {\bibinfo  {journal} {Elife}\ }\textbf {\bibinfo
  {volume} {5}},\ \bibinfo {pages} {e17267} (\bibinfo {year}
  {2016})}\BibitemShut {NoStop}%
\bibitem [{\citenamefont {Sejnowski}\ and\ \citenamefont
  {Destexhe}(2000)}]{sejnowski2000we}%
  \BibitemOpen
  \bibfield  {author} {\bibinfo {author} {\bibfnamefont {T.~J.}\ \bibnamefont
  {Sejnowski}}\ and\ \bibinfo {author} {\bibfnamefont {A.}~\bibnamefont
  {Destexhe}},\ }\bibfield  {title} {\bibinfo {title} {Why do we sleep?},\
  }\href@noop {} {\bibfield  {journal} {\bibinfo  {journal} {Brain research}\
  }\textbf {\bibinfo {volume} {886}},\ \bibinfo {pages} {208} (\bibinfo {year}
  {2000})}\BibitemShut {NoStop}%
\bibitem [{\citenamefont {Girard}\ \emph {et~al.}(2001)\citenamefont {Girard},
  \citenamefont {Hup{\'e}},\ and\ \citenamefont
  {Bullier}}]{girard2001feedforward}%
  \BibitemOpen
  \bibfield  {author} {\bibinfo {author} {\bibfnamefont {P.}~\bibnamefont
  {Girard}}, \bibinfo {author} {\bibfnamefont {J.~M.}\ \bibnamefont
  {Hup{\'e}}},\ and\ \bibinfo {author} {\bibfnamefont {J.}~\bibnamefont
  {Bullier}},\ }\bibfield  {title} {\bibinfo {title} {Feedforward and feedback
  connections between areas v1 and v2 of the monkey have similar rapid
  conduction velocities},\ }\href@noop {} {\bibfield  {journal} {\bibinfo
  {journal} {Journal of Neurophysiology}\ }\textbf {\bibinfo {volume} {85}},\
  \bibinfo {pages} {1328} (\bibinfo {year} {2001})}\BibitemShut {NoStop}%
\bibitem [{\citenamefont {Muller}\ \emph {et~al.}(2021)\citenamefont {Muller},
  \citenamefont {Min\'a\ifmmode~\check{c}\else \v{c}\fi{}},\ and\ \citenamefont
  {Nguyen}}]{muller2021algebraic}%
  \BibitemOpen
  \bibfield  {author} {\bibinfo {author} {\bibfnamefont {L.}~\bibnamefont
  {Muller}}, \bibinfo {author} {\bibfnamefont {J.}~\bibnamefont
  {Min\'a\ifmmode~\check{c}\else \v{c}\fi{}}},\ and\ \bibinfo {author}
  {\bibfnamefont {T.~T.}\ \bibnamefont {Nguyen}},\ }\bibfield  {title}
  {\bibinfo {title} {Algebraic approach to the kuramoto model},\ }\href
  {https://doi.org/10.1103/PhysRevE.104.L022201} {\bibfield  {journal}
  {\bibinfo  {journal} {Physical Review E}\ }\textbf {\bibinfo {volume}
  {104}},\ \bibinfo {pages} {L022201} (\bibinfo {year} {2021})}\BibitemShut
  {NoStop}%
\bibitem [{\citenamefont {Atay}(2010)}]{atay2010complex}%
  \BibitemOpen
  \bibfield  {author} {\bibinfo {author} {\bibfnamefont {F.~M.}\ \bibnamefont
  {Atay}},\ }\href@noop {} {\emph {\bibinfo {title} {Complex time-delay
  systems: theory and applications}}}\ (\bibinfo  {publisher} {Springer},\
  \bibinfo {year} {2010})\BibitemShut {NoStop}%
\bibitem [{\citenamefont {Schr{\"o}ter}\ \emph {et~al.}(2017)\citenamefont
  {Schr{\"o}ter}, \citenamefont {Paulsen},\ and\ \citenamefont
  {Bullmore}}]{schroter2017micro}%
  \BibitemOpen
  \bibfield  {author} {\bibinfo {author} {\bibfnamefont {M.}~\bibnamefont
  {Schr{\"o}ter}}, \bibinfo {author} {\bibfnamefont {O.}~\bibnamefont
  {Paulsen}},\ and\ \bibinfo {author} {\bibfnamefont {E.~T.}\ \bibnamefont
  {Bullmore}},\ }\bibfield  {title} {\bibinfo {title} {Micro-connectomics:
  probing the organization of neuronal networks at the cellular scale},\
  }\href@noop {} {\bibfield  {journal} {\bibinfo  {journal} {Nature Reviews
  Neuroscience}\ }\textbf {\bibinfo {volume} {18}},\ \bibinfo {pages} {131}
  (\bibinfo {year} {2017})}\BibitemShut {NoStop}%
\bibitem [{\citenamefont {Papachristodoulou}\ \emph {et~al.}(2010)\citenamefont
  {Papachristodoulou}, \citenamefont {Jadbabaie},\ and\ \citenamefont
  {M{\"u}nz}}]{papachristodoulou2010effects}%
  \BibitemOpen
  \bibfield  {author} {\bibinfo {author} {\bibfnamefont {A.}~\bibnamefont
  {Papachristodoulou}}, \bibinfo {author} {\bibfnamefont {A.}~\bibnamefont
  {Jadbabaie}},\ and\ \bibinfo {author} {\bibfnamefont {U.}~\bibnamefont
  {M{\"u}nz}},\ }\bibfield  {title} {\bibinfo {title} {Effects of delay in
  multi-agent consensus and oscillator synchronization},\ }\href@noop {}
  {\bibfield  {journal} {\bibinfo  {journal} {IEEE transactions on automatic
  control}\ }\textbf {\bibinfo {volume} {55}},\ \bibinfo {pages} {1471}
  (\bibinfo {year} {2010})}\BibitemShut {NoStop}%
\bibitem [{\citenamefont {Liao}\ and\ \citenamefont
  {Wang}(2000)}]{liao2000global}%
  \BibitemOpen
  \bibfield  {author} {\bibinfo {author} {\bibfnamefont {T.}~\bibnamefont
  {Liao}}\ and\ \bibinfo {author} {\bibfnamefont {F.}~\bibnamefont {Wang}},\
  }\bibfield  {title} {\bibinfo {title} {Global stability for cellular neural
  networks with time delay},\ }\href@noop {} {\bibfield  {journal} {\bibinfo
  {journal} {IEEE Transactions on neural networks}\ }\textbf {\bibinfo {volume}
  {11}},\ \bibinfo {pages} {1481} (\bibinfo {year} {2000})}\BibitemShut
  {NoStop}%
\bibitem [{\citenamefont {Tewarie}\ \emph {et~al.}(2019)\citenamefont
  {Tewarie}, \citenamefont {Abeysuriya}, \citenamefont {Byrne}, \citenamefont
  {O'Neill}, \citenamefont {Sotiropoulos}, \citenamefont {Brookes},\ and\
  \citenamefont {Coombes}}]{tewarie2019spatially}%
  \BibitemOpen
  \bibfield  {author} {\bibinfo {author} {\bibfnamefont {P.}~\bibnamefont
  {Tewarie}}, \bibinfo {author} {\bibfnamefont {R.}~\bibnamefont {Abeysuriya}},
  \bibinfo {author} {\bibfnamefont {A.}~\bibnamefont {Byrne}}, \bibinfo
  {author} {\bibfnamefont {G.~C.}\ \bibnamefont {O'Neill}}, \bibinfo {author}
  {\bibfnamefont {S.~N.}\ \bibnamefont {Sotiropoulos}}, \bibinfo {author}
  {\bibfnamefont {M.~J.}\ \bibnamefont {Brookes}},\ and\ \bibinfo {author}
  {\bibfnamefont {S.}~\bibnamefont {Coombes}},\ }\bibfield  {title} {\bibinfo
  {title} {How do spatially distinct frequency specific meg networks emerge
  from one underlying structural connectome? the role of the structural
  eigenmodes},\ }\href@noop {} {\bibfield  {journal} {\bibinfo  {journal}
  {NeuroImage}\ }\textbf {\bibinfo {volume} {186}},\ \bibinfo {pages} {211}
  (\bibinfo {year} {2019})}\BibitemShut {NoStop}%
\bibitem [{\citenamefont {Petkoski}\ and\ \citenamefont
  {Jirsa}(2022)}]{petkoski2022normalizing}%
  \BibitemOpen
  \bibfield  {author} {\bibinfo {author} {\bibfnamefont {S.}~\bibnamefont
  {Petkoski}}\ and\ \bibinfo {author} {\bibfnamefont {V.~K.}\ \bibnamefont
  {Jirsa}},\ }\bibfield  {title} {\bibinfo {title} {Normalizing the brain
  connectome for communication through synchronization},\ }\href@noop {}
  {\bibfield  {journal} {\bibinfo  {journal} {Network Neuroscience}\ ,\
  \bibinfo {pages} {1}} (\bibinfo {year} {2022})}\BibitemShut {NoStop}%
\bibitem [{\citenamefont {Lee}\ \emph {et~al.}(2009)\citenamefont {Lee},
  \citenamefont {Ott},\ and\ \citenamefont {Antonsen}}]{lee2009large}%
  \BibitemOpen
  \bibfield  {author} {\bibinfo {author} {\bibfnamefont {W.~S.}\ \bibnamefont
  {Lee}}, \bibinfo {author} {\bibfnamefont {E.}~\bibnamefont {Ott}},\ and\
  \bibinfo {author} {\bibfnamefont {T.~M.}\ \bibnamefont {Antonsen}},\
  }\bibfield  {title} {\bibinfo {title} {Large coupled oscillator systems with
  heterogeneous interaction delays},\ }\href@noop {} {\bibfield  {journal}
  {\bibinfo  {journal} {Physical Review Letters}\ }\textbf {\bibinfo {volume}
  {103}},\ \bibinfo {pages} {044101} (\bibinfo {year} {2009})}\BibitemShut
  {NoStop}%
\bibitem [{\citenamefont {Papachristodoulou}\ and\ \citenamefont
  {Jadbabaie}(2006)}]{papachristodoulou2006synchonization}%
  \BibitemOpen
  \bibfield  {author} {\bibinfo {author} {\bibfnamefont {A.}~\bibnamefont
  {Papachristodoulou}}\ and\ \bibinfo {author} {\bibfnamefont {A.}~\bibnamefont
  {Jadbabaie}},\ }\bibfield  {title} {\bibinfo {title} {Synchonization in
  oscillator networks with heterogeneous delays, switching topologies and
  nonlinear dynamics},\ }in\ \href@noop {} {\emph {\bibinfo {booktitle}
  {Proceedings of the 45th IEEE Conference on Decision and Control}}}\
  (\bibinfo {organization} {IEEE},\ \bibinfo {year} {2006})\ pp.\ \bibinfo
  {pages} {4307--4312}\BibitemShut {NoStop}%
\bibitem [{\citenamefont {Roberts}\ \emph {et~al.}(2019)\citenamefont
  {Roberts}, \citenamefont {Gollo}, \citenamefont {Abeysuriya}, \citenamefont
  {Roberts}, \citenamefont {Mitchell}, \citenamefont {Woolrich},\ and\
  \citenamefont {Breakspear}}]{roberts2019metastable}%
  \BibitemOpen
  \bibfield  {author} {\bibinfo {author} {\bibfnamefont {J.~A.}\ \bibnamefont
  {Roberts}}, \bibinfo {author} {\bibfnamefont {L.~L.}\ \bibnamefont {Gollo}},
  \bibinfo {author} {\bibfnamefont {R.~G.}\ \bibnamefont {Abeysuriya}},
  \bibinfo {author} {\bibfnamefont {G.}~\bibnamefont {Roberts}}, \bibinfo
  {author} {\bibfnamefont {P.~B.}\ \bibnamefont {Mitchell}}, \bibinfo {author}
  {\bibfnamefont {M.~W.}\ \bibnamefont {Woolrich}},\ and\ \bibinfo {author}
  {\bibfnamefont {M.}~\bibnamefont {Breakspear}},\ }\bibfield  {title}
  {\bibinfo {title} {Metastable brain waves},\ }\href@noop {} {\bibfield
  {journal} {\bibinfo  {journal} {Nature Communications}\ }\textbf {\bibinfo
  {volume} {10}},\ \bibinfo {pages} {1} (\bibinfo {year} {2019})}\BibitemShut
  {NoStop}%
\bibitem [{\citenamefont {Kuramoto}(1975)}]{kuramoto1975self}%
  \BibitemOpen
  \bibfield  {author} {\bibinfo {author} {\bibfnamefont {Y.}~\bibnamefont
  {Kuramoto}},\ }\bibfield  {title} {\bibinfo {title} {Self-entrainment of a
  population of coupled non-linear oscillators},\ }in\ \href@noop {} {\emph
  {\bibinfo {booktitle} {International symposium on mathematical problems in
  theoretical physics}}}\ (\bibinfo {organization} {Springer},\ \bibinfo {year}
  {1975})\ pp.\ \bibinfo {pages} {420--422}\BibitemShut {NoStop}%
\bibitem [{\citenamefont {Rodrigues}\ \emph {et~al.}(2016)\citenamefont
  {Rodrigues}, \citenamefont {Peron}, \citenamefont {Ji},\ and\ \citenamefont
  {Kurths}}]{rodrigues2016kuramoto}%
  \BibitemOpen
  \bibfield  {author} {\bibinfo {author} {\bibfnamefont {F.~A.}\ \bibnamefont
  {Rodrigues}}, \bibinfo {author} {\bibfnamefont {T.~K. D.~M.}\ \bibnamefont
  {Peron}}, \bibinfo {author} {\bibfnamefont {P.}~\bibnamefont {Ji}},\ and\
  \bibinfo {author} {\bibfnamefont {J.}~\bibnamefont {Kurths}},\ }\bibfield
  {title} {\bibinfo {title} {The kuramoto model in complex networks},\
  }\href@noop {} {\bibfield  {journal} {\bibinfo  {journal} {Physics Reports}\
  }\textbf {\bibinfo {volume} {610}},\ \bibinfo {pages} {1} (\bibinfo {year}
  {2016})}\BibitemShut {NoStop}%
\bibitem [{\citenamefont {Acebr{\'o}n}\ \emph {et~al.}(2005)\citenamefont
  {Acebr{\'o}n}, \citenamefont {Bonilla}, \citenamefont {Vicente},
  \citenamefont {Ritort},\ and\ \citenamefont {Spigler}}]{acebron2005kuramoto}%
  \BibitemOpen
  \bibfield  {author} {\bibinfo {author} {\bibfnamefont {J.~A.}\ \bibnamefont
  {Acebr{\'o}n}}, \bibinfo {author} {\bibfnamefont {L.~L.}\ \bibnamefont
  {Bonilla}}, \bibinfo {author} {\bibfnamefont {C.~J.~P.}\ \bibnamefont
  {Vicente}}, \bibinfo {author} {\bibfnamefont {F.}~\bibnamefont {Ritort}},\
  and\ \bibinfo {author} {\bibfnamefont {R.}~\bibnamefont {Spigler}},\
  }\bibfield  {title} {\bibinfo {title} {The kuramoto model: A simple paradigm
  for synchronization phenomena},\ }\href@noop {} {\bibfield  {journal}
  {\bibinfo  {journal} {Reviews of Modern Physics}\ }\textbf {\bibinfo {volume}
  {77}},\ \bibinfo {pages} {137} (\bibinfo {year} {2005})}\BibitemShut
  {NoStop}%
\bibitem [{\citenamefont {Jeong}\ \emph {et~al.}(2002)\citenamefont {Jeong},
  \citenamefont {Ko},\ and\ \citenamefont {Moon}}]{jeong2002time}%
  \BibitemOpen
  \bibfield  {author} {\bibinfo {author} {\bibfnamefont {S.~O.}\ \bibnamefont
  {Jeong}}, \bibinfo {author} {\bibfnamefont {T.~W.}\ \bibnamefont {Ko}},\ and\
  \bibinfo {author} {\bibfnamefont {H.~T.}\ \bibnamefont {Moon}},\ }\bibfield
  {title} {\bibinfo {title} {Time-delayed spatial patterns in a two-dimensional
  array of coupled oscillators},\ }\href@noop {} {\bibfield  {journal}
  {\bibinfo  {journal} {Physical Review Letters}\ }\textbf {\bibinfo {volume}
  {89}},\ \bibinfo {pages} {154104} (\bibinfo {year} {2002})}\BibitemShut
  {NoStop}%
\bibitem [{\citenamefont {Ko}\ and\ \citenamefont
  {Ermentrout}(2007)}]{ko2007effects}%
  \BibitemOpen
  \bibfield  {author} {\bibinfo {author} {\bibfnamefont {T.~W.}\ \bibnamefont
  {Ko}}\ and\ \bibinfo {author} {\bibfnamefont {G.~B.}\ \bibnamefont
  {Ermentrout}},\ }\bibfield  {title} {\bibinfo {title} {Effects of axonal time
  delay on synchronization and wave formation in sparsely coupled neuronal
  oscillators},\ }\href@noop {} {\bibfield  {journal} {\bibinfo  {journal}
  {Physical Review E}\ }\textbf {\bibinfo {volume} {76}},\ \bibinfo {pages}
  {056206} (\bibinfo {year} {2007})}\BibitemShut {NoStop}%
\bibitem [{\citenamefont {Petkoski}\ \emph {et~al.}(2016)\citenamefont
  {Petkoski}, \citenamefont {Spiegler}, \citenamefont {Proix}, \citenamefont
  {Aram}, \citenamefont {Temprado},\ and\ \citenamefont
  {Jirsa}}]{petkoski2016heterogeneity}%
  \BibitemOpen
  \bibfield  {author} {\bibinfo {author} {\bibfnamefont {S.}~\bibnamefont
  {Petkoski}}, \bibinfo {author} {\bibfnamefont {A.}~\bibnamefont {Spiegler}},
  \bibinfo {author} {\bibfnamefont {T.}~\bibnamefont {Proix}}, \bibinfo
  {author} {\bibfnamefont {P.}~\bibnamefont {Aram}}, \bibinfo {author}
  {\bibfnamefont {J.}~\bibnamefont {Temprado}},\ and\ \bibinfo {author}
  {\bibfnamefont {V.~K.}\ \bibnamefont {Jirsa}},\ }\bibfield  {title} {\bibinfo
  {title} {Heterogeneity of time delays determines synchronization of coupled
  oscillators},\ }\href@noop {} {\bibfield  {journal} {\bibinfo  {journal}
  {Physical Review E}\ }\textbf {\bibinfo {volume} {94}},\ \bibinfo {pages}
  {012209} (\bibinfo {year} {2016})}\BibitemShut {NoStop}%
\bibitem [{\citenamefont {Strogatz}(2000)}]{strogatz2000kuramoto}%
  \BibitemOpen
  \bibfield  {author} {\bibinfo {author} {\bibfnamefont {S.~H.}\ \bibnamefont
  {Strogatz}},\ }\bibfield  {title} {\bibinfo {title} {From kuramoto to
  crawford: exploring the onset of synchronization in populations of coupled
  oscillators},\ }\href@noop {} {\bibfield  {journal} {\bibinfo  {journal}
  {Physica D: Nonlinear Phenomena}\ }\textbf {\bibinfo {volume} {143}},\
  \bibinfo {pages} {1} (\bibinfo {year} {2000})}\BibitemShut {NoStop}%
\bibitem [{\citenamefont {Arenas}\ \emph {et~al.}(2008)\citenamefont {Arenas},
  \citenamefont {Diaz-Guilera}, \citenamefont {Kurths}, \citenamefont
  {Moreno},\ and\ \citenamefont {Zhou}}]{arenas2008synchronization}%
  \BibitemOpen
  \bibfield  {author} {\bibinfo {author} {\bibfnamefont {A.}~\bibnamefont
  {Arenas}}, \bibinfo {author} {\bibfnamefont {A.}~\bibnamefont
  {Diaz-Guilera}}, \bibinfo {author} {\bibfnamefont {J.}~\bibnamefont
  {Kurths}}, \bibinfo {author} {\bibfnamefont {Y.}~\bibnamefont {Moreno}},\
  and\ \bibinfo {author} {\bibfnamefont {C.}~\bibnamefont {Zhou}},\ }\bibfield
  {title} {\bibinfo {title} {Synchronization in complex networks},\ }\href@noop
  {} {\bibfield  {journal} {\bibinfo  {journal} {Physics Reports}\ }\textbf
  {\bibinfo {volume} {469}},\ \bibinfo {pages} {93} (\bibinfo {year}
  {2008})}\BibitemShut {NoStop}%
\bibitem [{\citenamefont {Muller}\ \emph {et~al.}(2018)\citenamefont {Muller},
  \citenamefont {Chavane}, \citenamefont {Reynolds},\ and\ \citenamefont
  {Sejnowski}}]{muller2018cortical}%
  \BibitemOpen
  \bibfield  {author} {\bibinfo {author} {\bibfnamefont {L.}~\bibnamefont
  {Muller}}, \bibinfo {author} {\bibfnamefont {F.}~\bibnamefont {Chavane}},
  \bibinfo {author} {\bibfnamefont {J.}~\bibnamefont {Reynolds}},\ and\
  \bibinfo {author} {\bibfnamefont {T.~J.}\ \bibnamefont {Sejnowski}},\
  }\bibfield  {title} {\bibinfo {title} {Cortical travelling waves: mechanisms
  and computational principles},\ }\href@noop {} {\bibfield  {journal}
  {\bibinfo  {journal} {Nature Reviews Neuroscience}\ }\textbf {\bibinfo
  {volume} {19}},\ \bibinfo {pages} {255} (\bibinfo {year} {2018})}\BibitemShut
  {NoStop}%
\bibitem [{\citenamefont {Davis}\ \emph {et~al.}(2021)\citenamefont {Davis},
  \citenamefont {Benigno}, \citenamefont {Fletterman}, \citenamefont
  {Desbordes}, \citenamefont {Steward}, \citenamefont {Sejnowski},
  \citenamefont {Reynolds},\ and\ \citenamefont
  {Muller}}]{davis2021spontaneous}%
  \BibitemOpen
  \bibfield  {author} {\bibinfo {author} {\bibfnamefont {Z.~W.}\ \bibnamefont
  {Davis}}, \bibinfo {author} {\bibfnamefont {G.~B.}\ \bibnamefont {Benigno}},
  \bibinfo {author} {\bibfnamefont {C.}~\bibnamefont {Fletterman}}, \bibinfo
  {author} {\bibfnamefont {T.}~\bibnamefont {Desbordes}}, \bibinfo {author}
  {\bibfnamefont {C.}~\bibnamefont {Steward}}, \bibinfo {author} {\bibfnamefont
  {T.~J.}\ \bibnamefont {Sejnowski}}, \bibinfo {author} {\bibfnamefont
  {J.}~\bibnamefont {Reynolds}},\ and\ \bibinfo {author} {\bibfnamefont
  {L.}~\bibnamefont {Muller}},\ }\bibfield  {title} {\bibinfo {title}
  {Spontaneous traveling waves naturally emerge from horizontal fiber time
  delays and travel through locally asynchronous-irregular states},\
  }\href@noop {} {\bibfield  {journal} {\bibinfo  {journal} {Nature
  Communications}\ }\textbf {\bibinfo {volume} {12}},\ \bibinfo {pages} {1}
  (\bibinfo {year} {2021})}\BibitemShut {NoStop}%
\bibitem [{\citenamefont {Breakspear}\ \emph {et~al.}(2010)\citenamefont
  {Breakspear}, \citenamefont {Heitmann},\ and\ \citenamefont
  {Daffertshofer}}]{breakspear2010generative}%
  \BibitemOpen
  \bibfield  {author} {\bibinfo {author} {\bibfnamefont {M.}~\bibnamefont
  {Breakspear}}, \bibinfo {author} {\bibfnamefont {S.}~\bibnamefont
  {Heitmann}},\ and\ \bibinfo {author} {\bibfnamefont {A.}~\bibnamefont
  {Daffertshofer}},\ }\bibfield  {title} {\bibinfo {title} {Generative models
  of cortical oscillations: neurobiological implications of the kuramoto
  model},\ }\href@noop {} {\bibfield  {journal} {\bibinfo  {journal} {Frontiers
  in Human Neuroscience}\ }\textbf {\bibinfo {volume} {4}},\ \bibinfo {pages}
  {190} (\bibinfo {year} {2010})}\BibitemShut {NoStop}%
\bibitem [{\citenamefont {Cabral}\ \emph {et~al.}(2011)\citenamefont {Cabral},
  \citenamefont {Hugues}, \citenamefont {Sporns},\ and\ \citenamefont
  {Deco}}]{cabral2011role}%
  \BibitemOpen
  \bibfield  {author} {\bibinfo {author} {\bibfnamefont {J.}~\bibnamefont
  {Cabral}}, \bibinfo {author} {\bibfnamefont {E.}~\bibnamefont {Hugues}},
  \bibinfo {author} {\bibfnamefont {O.}~\bibnamefont {Sporns}},\ and\ \bibinfo
  {author} {\bibfnamefont {G.}~\bibnamefont {Deco}},\ }\bibfield  {title}
  {\bibinfo {title} {Role of local network oscillations in resting-state
  functional connectivity},\ }\href@noop {} {\bibfield  {journal} {\bibinfo
  {journal} {Neuroimage}\ }\textbf {\bibinfo {volume} {57}},\ \bibinfo {pages}
  {130} (\bibinfo {year} {2011})}\BibitemShut {NoStop}%
\bibitem [{\citenamefont {Choi}\ and\ \citenamefont
  {Mihalas}(2019)}]{choi2019synchronization}%
  \BibitemOpen
  \bibfield  {author} {\bibinfo {author} {\bibfnamefont {H.}~\bibnamefont
  {Choi}}\ and\ \bibinfo {author} {\bibfnamefont {S.}~\bibnamefont {Mihalas}},\
  }\bibfield  {title} {\bibinfo {title} {Synchronization dependent on spatial
  structures of a mesoscopic whole-brain network},\ }\href@noop {} {\bibfield
  {journal} {\bibinfo  {journal} {PLoS Computational Biology}\ }\textbf
  {\bibinfo {volume} {15}},\ \bibinfo {pages} {e1006978} (\bibinfo {year}
  {2019})}\BibitemShut {NoStop}%
\bibitem [{fig({\natexlab{a}})}]{fig_s1}%
  \BibitemOpen
  \href@noop {} {\bibinfo {title} {See Supplemental Material for Fig. S1}}
  \BibitemShut {NoStop}%
\bibitem [{\citenamefont {Budzinski}\ \emph {et~al.}(2022)\citenamefont
  {Budzinski}, \citenamefont {Nguyen}, \citenamefont {{\DJ}o{\`a}n},
  \citenamefont {Min{\'a}{\v{c}}}, \citenamefont {Sejnowski},\ and\
  \citenamefont {Muller}}]{budzinski2022geometry}%
  \BibitemOpen
  \bibfield  {author} {\bibinfo {author} {\bibfnamefont {R.~C.}\ \bibnamefont
  {Budzinski}}, \bibinfo {author} {\bibfnamefont {T.~T.}\ \bibnamefont
  {Nguyen}}, \bibinfo {author} {\bibfnamefont {J.}~\bibnamefont
  {{\DJ}o{\`a}n}}, \bibinfo {author} {\bibfnamefont {J.}~\bibnamefont
  {Min{\'a}{\v{c}}}}, \bibinfo {author} {\bibfnamefont {T.~J.}\ \bibnamefont
  {Sejnowski}},\ and\ \bibinfo {author} {\bibfnamefont {L.~E.}\ \bibnamefont
  {Muller}},\ }\bibfield  {title} {\bibinfo {title} {Geometry unites synchrony,
  chimeras, and waves in nonlinear oscillator networks},\ }\href@noop {}
  {\bibfield  {journal} {\bibinfo  {journal} {Chaos: An Interdisciplinary
  Journal of Nonlinear Science}\ }\textbf {\bibinfo {volume} {32}},\ \bibinfo
  {pages} {031104} (\bibinfo {year} {2022})}\BibitemShut {NoStop}%
\bibitem [{\citenamefont {Yeung}\ and\ \citenamefont
  {Strogatz}(1999)}]{yeung1999time}%
  \BibitemOpen
  \bibfield  {author} {\bibinfo {author} {\bibfnamefont {M.~K.~S.}\
  \bibnamefont {Yeung}}\ and\ \bibinfo {author} {\bibfnamefont {S.~H.}\
  \bibnamefont {Strogatz}},\ }\bibfield  {title} {\bibinfo {title} {Time delay
  in the kuramoto model of coupled oscillators},\ }\href@noop {} {\bibfield
  {journal} {\bibinfo  {journal} {Physical Review Letters}\ }\textbf {\bibinfo
  {volume} {82}},\ \bibinfo {pages} {648} (\bibinfo {year} {1999})}\BibitemShut
  {NoStop}%
\bibitem [{\citenamefont {Laing}(2016)}]{laing2016travelling}%
  \BibitemOpen
  \bibfield  {author} {\bibinfo {author} {\bibfnamefont {C.~R.}\ \bibnamefont
  {Laing}},\ }\bibfield  {title} {\bibinfo {title} {Travelling waves in arrays
  of delay-coupled phase oscillators},\ }\href@noop {} {\bibfield  {journal}
  {\bibinfo  {journal} {Chaos: An Interdisciplinary Journal of Nonlinear
  Science}\ }\textbf {\bibinfo {volume} {26}},\ \bibinfo {pages} {094802}
  (\bibinfo {year} {2016})}\BibitemShut {NoStop}%
\bibitem [{\citenamefont {Timms}\ and\ \citenamefont
  {English}(2014)}]{timms2014synchronization}%
  \BibitemOpen
  \bibfield  {author} {\bibinfo {author} {\bibfnamefont {L.}~\bibnamefont
  {Timms}}\ and\ \bibinfo {author} {\bibfnamefont {L.~Q.}\ \bibnamefont
  {English}},\ }\bibfield  {title} {\bibinfo {title} {Synchronization in
  phase-coupled kuramoto oscillator networks with axonal delay and synaptic
  plasticity},\ }\href@noop {} {\bibfield  {journal} {\bibinfo  {journal}
  {Physical Review E}\ }\textbf {\bibinfo {volume} {89}},\ \bibinfo {pages}
  {032906} (\bibinfo {year} {2014})}\BibitemShut {NoStop}%
\bibitem [{\citenamefont {Ermentrout}\ and\ \citenamefont
  {Ko}(2009)}]{ermentrout2009delays}%
  \BibitemOpen
  \bibfield  {author} {\bibinfo {author} {\bibfnamefont {B.}~\bibnamefont
  {Ermentrout}}\ and\ \bibinfo {author} {\bibfnamefont {T.~W.}\ \bibnamefont
  {Ko}},\ }\bibfield  {title} {\bibinfo {title} {Delays and weakly coupled
  neuronal oscillators},\ }\href@noop {} {\bibfield  {journal} {\bibinfo
  {journal} {Philosophical Transactions of the Royal Society A: Mathematical,
  Physical and Engineering Sciences}\ }\textbf {\bibinfo {volume} {367}},\
  \bibinfo {pages} {1097} (\bibinfo {year} {2009})}\BibitemShut {NoStop}%
\bibitem [{\citenamefont {Ko}\ \emph {et~al.}(2004)\citenamefont {Ko},
  \citenamefont {Jeong},\ and\ \citenamefont {Moon}}]{ko2004wave}%
  \BibitemOpen
  \bibfield  {author} {\bibinfo {author} {\bibfnamefont {T.~W.}\ \bibnamefont
  {Ko}}, \bibinfo {author} {\bibfnamefont {S.~O.}\ \bibnamefont {Jeong}},\ and\
  \bibinfo {author} {\bibfnamefont {H.~T.}\ \bibnamefont {Moon}},\ }\bibfield
  {title} {\bibinfo {title} {Wave formation by time delays in randomly coupled
  oscillators},\ }\href@noop {} {\bibfield  {journal} {\bibinfo  {journal}
  {Physical Review E}\ }\textbf {\bibinfo {volume} {69}},\ \bibinfo {pages}
  {056106} (\bibinfo {year} {2004})}\BibitemShut {NoStop}%
\bibitem [{\citenamefont {Grabow}\ \emph {et~al.}(2010)\citenamefont {Grabow},
  \citenamefont {Hill}, \citenamefont {Grosskinsky},\ and\ \citenamefont
  {Timme}}]{grabow2010small}%
  \BibitemOpen
  \bibfield  {author} {\bibinfo {author} {\bibfnamefont {C.}~\bibnamefont
  {Grabow}}, \bibinfo {author} {\bibfnamefont {S.~M.}\ \bibnamefont {Hill}},
  \bibinfo {author} {\bibfnamefont {S.}~\bibnamefont {Grosskinsky}},\ and\
  \bibinfo {author} {\bibfnamefont {M.}~\bibnamefont {Timme}},\ }\bibfield
  {title} {\bibinfo {title} {Do small worlds synchronize fastest?},\
  }\href@noop {} {\bibfield  {journal} {\bibinfo  {journal} {EPL (Europhysics
  Letters)}\ }\textbf {\bibinfo {volume} {90}},\ \bibinfo {pages} {48002}
  (\bibinfo {year} {2010})}\BibitemShut {NoStop}%
\bibitem [{\citenamefont {Sorrentino}\ \emph {et~al.}(2016)\citenamefont
  {Sorrentino}, \citenamefont {Pecora}, \citenamefont {Hagerstrom},
  \citenamefont {Murphy},\ and\ \citenamefont {Roy}}]{sorrentino2016complete}%
  \BibitemOpen
  \bibfield  {author} {\bibinfo {author} {\bibfnamefont {F.}~\bibnamefont
  {Sorrentino}}, \bibinfo {author} {\bibfnamefont {L.~M.}\ \bibnamefont
  {Pecora}}, \bibinfo {author} {\bibfnamefont {A.~M.}\ \bibnamefont
  {Hagerstrom}}, \bibinfo {author} {\bibfnamefont {T.~E.}\ \bibnamefont
  {Murphy}},\ and\ \bibinfo {author} {\bibfnamefont {R.}~\bibnamefont {Roy}},\
  }\bibfield  {title} {\bibinfo {title} {Complete characterization of the
  stability of cluster synchronization in complex dynamical networks},\
  }\href@noop {} {\bibfield  {journal} {\bibinfo  {journal} {Science Advances}\
  }\textbf {\bibinfo {volume} {2}},\ \bibinfo {pages} {e1501737} (\bibinfo
  {year} {2016})}\BibitemShut {NoStop}%
\bibitem [{\citenamefont {Sugitani}\ \emph {et~al.}(2021)\citenamefont
  {Sugitani}, \citenamefont {Zhang},\ and\ \citenamefont
  {Motter}}]{sugitani2021synchronizing}%
  \BibitemOpen
  \bibfield  {author} {\bibinfo {author} {\bibfnamefont {Y.}~\bibnamefont
  {Sugitani}}, \bibinfo {author} {\bibfnamefont {Y.}~\bibnamefont {Zhang}},\
  and\ \bibinfo {author} {\bibfnamefont {A.~E.}\ \bibnamefont {Motter}},\
  }\bibfield  {title} {\bibinfo {title} {Synchronizing chaos with
  imperfections},\ }\href@noop {} {\bibfield  {journal} {\bibinfo  {journal}
  {Physical Review Letters}\ }\textbf {\bibinfo {volume} {126}},\ \bibinfo
  {pages} {164101} (\bibinfo {year} {2021})}\BibitemShut {NoStop}%
\bibitem [{\citenamefont {Davis}(1979)}]{davis1979}%
  \BibitemOpen
  \bibfield  {author} {\bibinfo {author} {\bibfnamefont {P.}~\bibnamefont
  {Davis}},\ }\href@noop {} {\emph {\bibinfo {title} {Circulant matrices}}}\
  (\bibinfo  {publisher} {John Wiley \& Sons},\ \bibinfo {year}
  {1979})\BibitemShut {NoStop}%
\bibitem [{fig({\natexlab{b}})}]{fig_s2}%
  \BibitemOpen
  \href@noop {} {\bibinfo {title} {See Supplemental Material for Fig. S2}}
  \BibitemShut {NoStop}%
\bibitem [{fig({\natexlab{c}})}]{fig_s3}%
  \BibitemOpen
  \href@noop {} {\bibinfo {title} {See Supplemental Material for Fig. S3}}
  \BibitemShut {NoStop}%
\bibitem [{mov({\natexlab{a}})}]{movie_s1}%
  \BibitemOpen
  \href@noop {} {\bibinfo {title} {See Supplemental Material for Movie S1}}
  \BibitemShut {NoStop}%
\bibitem [{mov({\natexlab{b}})}]{movie_s2}%
  \BibitemOpen
  \href@noop {} {\bibinfo {title} {See Supplemental Material for Movie S2}}
  \BibitemShut {NoStop}%
\bibitem [{\citenamefont {Hagmann}\ \emph {et~al.}(2008)\citenamefont
  {Hagmann}, \citenamefont {Cammoun}, \citenamefont {Gigandet}, \citenamefont
  {Meuli}, \citenamefont {Honey}, \citenamefont {Wedeen},\ and\ \citenamefont
  {Sporns}}]{hagmann2008mapping}%
  \BibitemOpen
  \bibfield  {author} {\bibinfo {author} {\bibfnamefont {P.}~\bibnamefont
  {Hagmann}}, \bibinfo {author} {\bibfnamefont {L.}~\bibnamefont {Cammoun}},
  \bibinfo {author} {\bibfnamefont {X.}~\bibnamefont {Gigandet}}, \bibinfo
  {author} {\bibfnamefont {R.}~\bibnamefont {Meuli}}, \bibinfo {author}
  {\bibfnamefont {C.~J.}\ \bibnamefont {Honey}}, \bibinfo {author}
  {\bibfnamefont {V.~J.}\ \bibnamefont {Wedeen}},\ and\ \bibinfo {author}
  {\bibfnamefont {O.}~\bibnamefont {Sporns}},\ }\bibfield  {title} {\bibinfo
  {title} {Mapping the structural core of human cerebral cortex},\ }\href@noop
  {} {\bibfield  {journal} {\bibinfo  {journal} {PLoS Biology}\ }\textbf
  {\bibinfo {volume} {6}},\ \bibinfo {pages} {e159} (\bibinfo {year}
  {2008})}\BibitemShut {NoStop}%
\bibitem [{\citenamefont {Muller}\ \emph {et~al.}(2014)\citenamefont {Muller},
  \citenamefont {Destexhe},\ and\ \citenamefont
  {Rudolph-Lilith}}]{muller2014brain}%
  \BibitemOpen
  \bibfield  {author} {\bibinfo {author} {\bibfnamefont {L.}~\bibnamefont
  {Muller}}, \bibinfo {author} {\bibfnamefont {A.}~\bibnamefont {Destexhe}},\
  and\ \bibinfo {author} {\bibfnamefont {M.}~\bibnamefont {Rudolph-Lilith}},\
  }\bibfield  {title} {\bibinfo {title} {Brain networks: small-worlds, after
  all?},\ }\href@noop {} {\bibfield  {journal} {\bibinfo  {journal} {New
  Journal of Physics}\ }\textbf {\bibinfo {volume} {16}},\ \bibinfo {pages}
  {105004} (\bibinfo {year} {2014})}\BibitemShut {NoStop}%
\bibitem [{\citenamefont {Swadlow}\ and\ \citenamefont
  {Waxman}(2012)}]{swadlow2012axonal}%
  \BibitemOpen
  \bibfield  {author} {\bibinfo {author} {\bibfnamefont {H.~A.}\ \bibnamefont
  {Swadlow}}\ and\ \bibinfo {author} {\bibfnamefont {S.~G.}\ \bibnamefont
  {Waxman}},\ }\bibfield  {title} {\bibinfo {title} {Axonal conduction
  delays},\ }\href@noop {} {\bibfield  {journal} {\bibinfo  {journal}
  {Scholarpedia}\ }\textbf {\bibinfo {volume} {7}},\ \bibinfo {pages} {1451}
  (\bibinfo {year} {2012})}\BibitemShut {NoStop}%
\bibitem [{mov({\natexlab{c}})}]{movie_s3}%
  \BibitemOpen
  \href@noop {} {\bibinfo {title} {See Supplemental Material for Movie S3}}
 \BibitemShut {NoStop}%
\bibitem [{mov({\natexlab{d}})}]{movie_s4}%
  \BibitemOpen
  \href@noop {} {\bibinfo {title} {See Supplemental Material for Movie S4}}
  \BibitemShut {NoStop}%
\bibitem [{\citenamefont {Litvina}\ \emph {et~al.}(2019)\citenamefont
  {Litvina}, \citenamefont {Adams}, \citenamefont {Barth}, \citenamefont
  {Bruchez}, \citenamefont {Carson}, \citenamefont {Chung}, \citenamefont
  {Dupre}, \citenamefont {Frank}, \citenamefont {Gates}, \citenamefont {Harris}
  \emph {et~al.}}]{litvina2019brain}%
  \BibitemOpen
  \bibfield  {author} {\bibinfo {author} {\bibfnamefont {E.}~\bibnamefont
  {Litvina}}, \bibinfo {author} {\bibfnamefont {A.}~\bibnamefont {Adams}},
  \bibinfo {author} {\bibfnamefont {A.}~\bibnamefont {Barth}}, \bibinfo
  {author} {\bibfnamefont {M.}~\bibnamefont {Bruchez}}, \bibinfo {author}
  {\bibfnamefont {J.}~\bibnamefont {Carson}}, \bibinfo {author} {\bibfnamefont
  {J.~E.}\ \bibnamefont {Chung}}, \bibinfo {author} {\bibfnamefont {K.~B.}\
  \bibnamefont {Dupre}}, \bibinfo {author} {\bibfnamefont {L.~M.}\ \bibnamefont
  {Frank}}, \bibinfo {author} {\bibfnamefont {K.~M.}\ \bibnamefont {Gates}},
  \bibinfo {author} {\bibfnamefont {K.~M.}\ \bibnamefont {Harris}}, \emph
  {et~al.},\ }\bibfield  {title} {\bibinfo {title} {Brain initiative:
  cutting-edge tools and resources for the community},\ }\href@noop {}
  {\bibfield  {journal} {\bibinfo  {journal} {Journal of Neuroscience}\
  }\textbf {\bibinfo {volume} {39}},\ \bibinfo {pages} {8275} (\bibinfo {year}
  {2019})}\BibitemShut {NoStop}%
\bibitem [{\citenamefont {Witthaut}\ \emph {et~al.}(2022)\citenamefont
  {Witthaut}, \citenamefont {Hellmann}, \citenamefont {Kurths}, \citenamefont
  {Kettemann}, \citenamefont {Meyer-Ortmanns},\ and\ \citenamefont
  {Timme}}]{witthaut2022collective}%
  \BibitemOpen
  \bibfield  {author} {\bibinfo {author} {\bibfnamefont {D.}~\bibnamefont
  {Witthaut}}, \bibinfo {author} {\bibfnamefont {F.}~\bibnamefont {Hellmann}},
  \bibinfo {author} {\bibfnamefont {J.}~\bibnamefont {Kurths}}, \bibinfo
  {author} {\bibfnamefont {S.}~\bibnamefont {Kettemann}}, \bibinfo {author}
  {\bibfnamefont {H.}~\bibnamefont {Meyer-Ortmanns}},\ and\ \bibinfo {author}
  {\bibfnamefont {M.}~\bibnamefont {Timme}},\ }\bibfield  {title} {\bibinfo
  {title} {Collective nonlinear dynamics and self-organization in decentralized
  power grids},\ }\href@noop {} {\bibfield  {journal} {\bibinfo  {journal}
  {Reviews of Modern Physics}\ }\textbf {\bibinfo {volume} {94}},\ \bibinfo
  {pages} {015005} (\bibinfo {year} {2022})}\BibitemShut {NoStop}%
\bibitem [{\citenamefont {Motter}\ \emph {et~al.}(2013)\citenamefont {Motter},
  \citenamefont {Myers}, \citenamefont {Anghel},\ and\ \citenamefont
  {Nishikawa}}]{motter2013spontaneous}%
  \BibitemOpen
  \bibfield  {author} {\bibinfo {author} {\bibfnamefont {A.~E.}\ \bibnamefont
  {Motter}}, \bibinfo {author} {\bibfnamefont {S.~A.}\ \bibnamefont {Myers}},
  \bibinfo {author} {\bibfnamefont {M.}~\bibnamefont {Anghel}},\ and\ \bibinfo
  {author} {\bibfnamefont {T.}~\bibnamefont {Nishikawa}},\ }\bibfield  {title}
  {\bibinfo {title} {Spontaneous synchrony in power-grid networks},\
  }\href@noop {} {\bibfield  {journal} {\bibinfo  {journal} {Nature Physics}\
  }\textbf {\bibinfo {volume} {9}},\ \bibinfo {pages} {191} (\bibinfo {year}
  {2013})}\BibitemShut {NoStop}%
\end{thebibliography}
%

\end{document}